\documentclass[11pt]{amsart}

\usepackage{amssymb}
\usepackage{epsfig}
\usepackage{comment}
\usepackage{amsmath}
\usepackage{subcaption}
\usepackage[section]{placeins}
\usepackage{mathrsfs}
\usepackage{graphicx}
\usepackage[notrig]{physics}
\usepackage{units}

\newtheorem{thm}{Theorem}[section]
\newtheorem{theorem}[thm]{Theorem}
\newtheorem{proposition}[thm]{Proposition}
\newtheorem{lemma}[thm]{Lemma}
\newtheorem{corollary}[thm]{Corollary}

\theoremstyle{definition}
\newtheorem{definition}[thm]{Definition}
\newtheorem{example}[thm]{Example}

\theoremstyle{remark}
\newtheorem{remark}[thm]{Remark}
\newcommand{\LM}[1]{\hbox{\vrule width.2pt \vbox to#1pt{\vfill \hrule width#1pt height.2pt}}}
\newcommand{\LL}{{\mathchoice{\,\LM7\,}{\,\LM7\,}{\,\LM5\,}{\,\LM{3.35}\,}}}

\newcommand{\vertiii}[1]{{\left\vert\kern-0.25ex\left\vert\kern-0.25ex\left\vert #1 
    \right\vert\kern-0.25ex\right\vert\kern-0.25ex\right\vert}}

\begin{document}
\title{Model-free Data-Driven Inference}

\author{S.~Conti}
\author{F.~Hoffmann}
\address{Institut f\"ur Angewandte Mathematik, Universit\"at Bonn}
\email{sergio.conti@uni-bonn.de, franca.hoffmann@hcm.uni-bonn.de}

\author{M.~Ortiz}
\address{Division of Engineering and Applied Science, Caltech; Hausdorff Center for Mathematics, Universit\"at Bonn}
\email{ortiz@caltech.edu}

\begin{abstract}
We present a model-free data-driven inference method that enables inferences on system outcomes to be derived directly from empirical data without the need for intervening modeling of any type, be it modeling of a material law or modeling of a prior distribution of material states.
We specifically consider physical systems with states characterized by points in a phase space determined by the governing field equations. We assume that the system is characterized by two likelihood measures: one $\mu_D$ measuring the likelihood of observing a material state in phase space; and another $\mu_E$ measuring the likelihood of states satisfying the field equations, possibly under random actuation. 
We introduce a notion of intersection between measures which can be interpreted to quantify the likelihood of system outcomes. We provide conditions under which the intersection can be characterized as the athermal limit $\mu_\infty$ of entropic regularizations $\mu_\beta$, or thermalizations, of the product measure $\mu = \mu_D\times \mu_E$ as $\beta \to +\infty$. We also supply conditions under which $\mu_\infty$ can be obtained as the athermal limit of carefully thermalized $(\mu_{h,\beta_h})$ sequences of empirical data sets $(\mu_h)$ approximating weakly an unknown likelihood function $\mu$. In particular, we find that the cooling sequence $\beta_h \to +\infty$ must be slow enough, corresponding to quenching, in order for the proper limit $\mu_\infty$ to be delivered. Finally, we derive explicit analytic expressions for expectations $\mathbb{E}[f]$ of outcomes $f$ that are explicit in the data, thus demonstrating the feasibility of the model-free data-driven paradigm as regards making convergent inferences directly from the data without recourse to intermediate modeling steps.
\end{abstract}

\maketitle

%

\section{Introduction}

The boundary value problems of continuum mechanics and mathematical physics have a precise structure that set them apart from other classes of problems (cf., e.~g., \cite{Truesdell:1960}). Thus, the governing field equations set forth hard constraints in the form of partial differential equations and attendant boundary conditions that are universal, i.~e., material independent, and free of epistemic uncertainty. However, in order to define well-posed boundary value problems the field equations must be closed through the specification of a material law, which is material specific and determined empirically (cf., e.~g., \cite{Truesdell:1965}). The classical approach to formulating material laws, or {\sl material identification}, relies on modeling to represent the available material data in some appropriate mathematical form, be it equations of state, kinetic and hereditary laws (such as in viscoelasticity) and other representations (cf., e.~g., \cite{Meyers:1994,Bower:2010}). For stochastic systems, the process of modeling is often compounded by the need to additionally model priors, e.~g., in the context of Bayesian inference (cf., e.~g., \cite[Section 2]{Dashti:2017}).

There is no general theory that enables, starting from empirical data, the identification of material models of an arbitrary degree of accuracy that are sure to converge, in some appropriate sense, to the exact but unknown material law as the volume of empirical data increases. In practice, {\sl ad hoc} parameterized functions are often fitted to the data by means of regression or some other form of parametric estimation (cf., e.~g., \cite{Bock:2019} for a review of recent developments centered on machine learning, Bayesian learning, manifold learning, model reduction and other approaches). Models necessarily rely heavily on heuristics and intuition and inevitably introduce biases and uncontrolled modeling errors. They can also result in a massive loss of information relative to that which is contained in the empirical data sets themselves. These uncertainties render material modeling {\sl ad hoc}, open-ended, ill-posed and a major limiting factor as regards the ability to make accurate and reliable inferences of the outcomes of physical systems.

The epochal advances in experimental science of the past two decades, including time- and space-resolved full-field microscopy, have transformed mathematical physics from a data-poor to a data-rich field, which raises a number of fundamental questions in theory and in practice. In particular, the present abundance of material data begs the question whether a direct connection between material data and predicted outcomes can be effected that altogether bypasses the traditional step of modeling material behavior, be it via material laws or prior distributions. A notional comparison between classical and model-free data-driven inference is:
\begin{equation*}
\begin{array}{cccccc}
    \text{Classical inference:}
    & \text{Data} & \to & \text{Model} & \to & \text{Prediction}
    \\
    \text{Model-Free Data-Driven inference:}
    & \text{Data} &  & \longrightarrow &  & \text{Prediction}
\end{array}
\end{equation*}
Evidently, the model-free data-driven paradigm is lossless, i.~e., it incurs no loss of information with respect to the data set; unbiased, i.~e., it requires no assumptions regarding variables or prior distribution of the data; and trivially modelling-error free, as it bypasses the classical step of building a material model altogether.

The present work is concerned with the formulation of one such model-free data-driven inference paradigm and with establishing its well-posedness and properties of convergence with respect to the data. We specifically consider physical systems with states characterized by points $z$ in a phase space $Z$ determined by the governing field equations. In the deterministic setting, cf.~Fig.~\ref{3rlASw} and Section~\ref{sec:det}, the physical field equations then have the effect of restricting the possible states of the system to an affine subspace $E$ of $Z$, which we refer to as the {\sl constraint set}. For instance, in solid mechanics, the phase space $Z$ is the space of strain and stress $(\epsilon,\sigma)$ over the body and the field equations are compatibility of strains and equilibrium of stresses, which may depend on external forcing and boundary conditions. The constraint subspace $E$ of admissible states is, therefore, the set of stress and strain fields that are compatible and in equilibrium with the applied loading. In addition, material behavior restricts the possible states $(\epsilon,\sigma)$ of the system to a {\sl material set} $D$ in $Z$. Often, the material set is local, i.~e., defined pointwise. For instance, for a local elastic material the material set has the representation $D = \{(\epsilon(x),\sigma(x)) \in D_{\rm loc} \subset Z_{\rm loc},\text{ for a.~e. } x\in\Omega\}$, where $Z_{\rm loc}$ is the phase-space of a single material point, $D_{\rm loc}$  is the local material set, e.~g., the graph of a local material law, and $\Omega$ the reference configuration. The classical deterministic solution set is, therefore, $D \cap E$, which is non-empty provided $D$ and $E$ satisfy appropriate closedness and transversality conditions \cite{conti2018data}.

\begin{figure}[ht]
\begin{center}
	\begin{subfigure}{0.35\textwidth}\caption{} \includegraphics[width=0.99\linewidth]{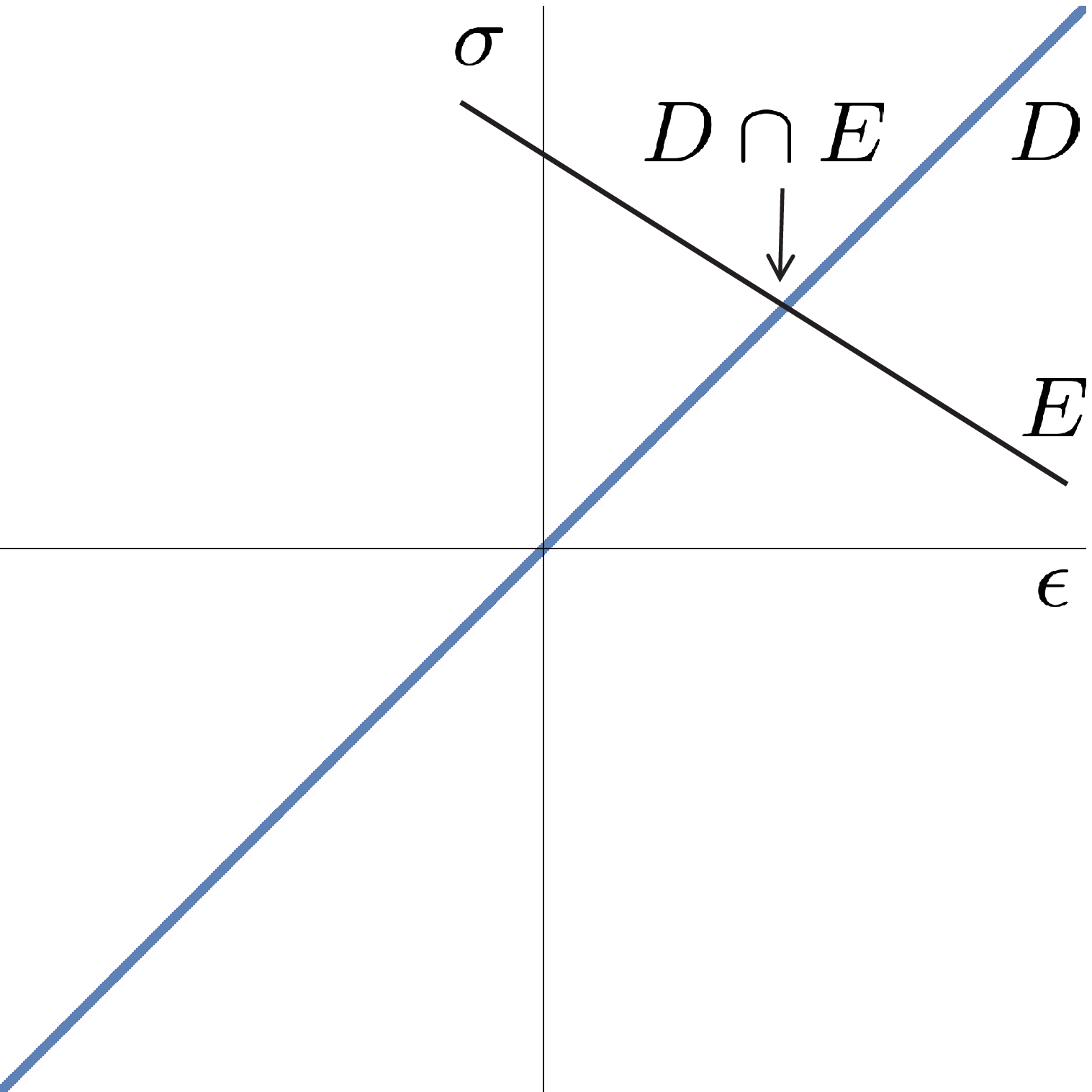}
	\end{subfigure}
    $\quad$
	\begin{subfigure}{0.35\textwidth}\caption{} \includegraphics[width=0.99\linewidth]{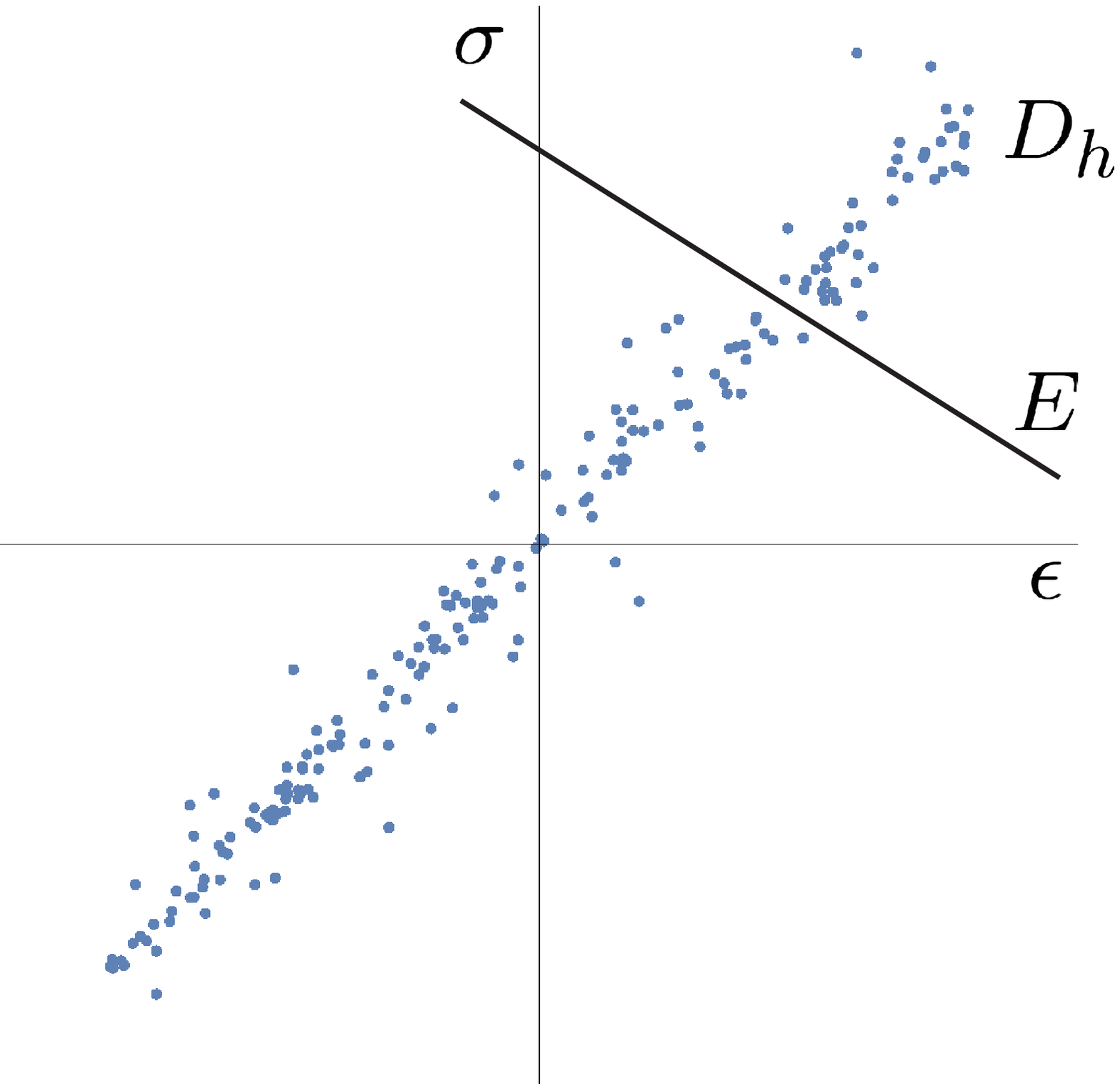}
	\end{subfigure}
    \caption{Deterministic case. a) Material data set for linear elasticity taking the form of an unbounded graph $D$, constraint set $E$ and classical solution $D\cap E$. b) Empirical material data set $D_h$ showing randomness due to measurement error and experimental scatter. Note this may result in an empty intersection between $D_h$ and $E$.} \label{3rlASw}
\end{center}
\end{figure}

Often, however, the material set $D$ is only known approximately through a sequence $D_h$ of approximating data sets, e.~g., consisting of empirical measurements, cf.~Fig.~\ref{3rlASw}b. In that case, the intersection $D_h \cap E$ is likely to be empty and $D_h$ fails to generate an approximating sequence of solutions in the classical sense. To circumvent this difficulty, \cite{kirchdoerfer2016data,conti2018data} proposed a Data-Driven (DD) regularization in which approximate solutions are identified with pairs of states $y_h \in D_h$, $z_h \in E$ such that some appropriate distance $d(y_h,z_h)$ is minimized in $Z \times Z$. Choosing again the example of solid mechanics for purposes of illustration, the data-driven solutions thus defined consist of a pair $z_h=(\epsilon_h,\sigma_h)$ and $y_h=(\epsilon_h', \sigma_h')$ of stress and strain fields, where the state $z_h$ is required to be in the admissible set $E$, i.~e., to consist of a compatible strain field and a stress field in equilibrium, whereas the state $y_h$ is required to be in the material set $D_h$. In a deterministic framework, solving the data-driven problem then entails minimizing an appropriate distance between the points $z_h=(\epsilon_h,\sigma_h)$ and $y_h=(\epsilon_h', \sigma_h')$. Appropriate notions of convergence of the data set $D_h \to D$  ensuring convergence of solutions, as well as related notions of relaxation in the infinite-dimensional setting, have been set forth in \cite{conti2018data,Conti:2020,Roger:2020}. We note that the paradigm is strictly data-driven and model-free in the sense that solutions are obtained, or approximated, directly from the data set without recourse to any intervening modeling of the data. Extensions of the approach, applications and follow-up work have spawned a sizeable engineering literature to date (cf., e.~g., \cite{Nguyen:2018,Ayensa:2018,Leygue:2018,Kanno:2018b,Zhou:2020,Gebhardt:2020a,Gebhardt:2020b} for a representative sample).

\begin{figure}[ht]
\begin{center}
	\begin{subfigure}{0.35\textwidth}\caption{} \includegraphics[width=0.99\linewidth]{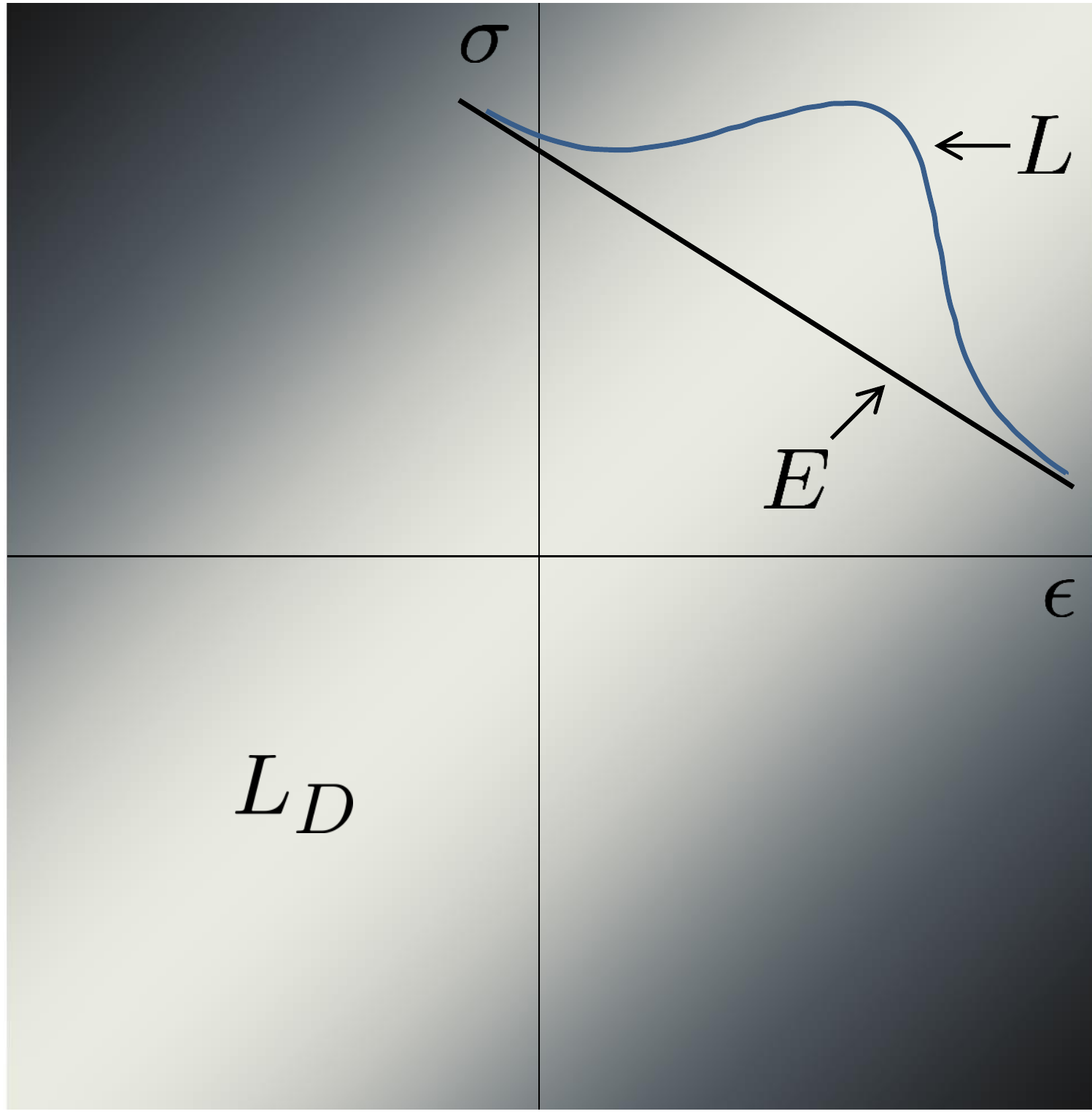}
	\end{subfigure}
    $\quad$
	\begin{subfigure}{0.35\textwidth}\caption{} \includegraphics[width=0.99\linewidth]{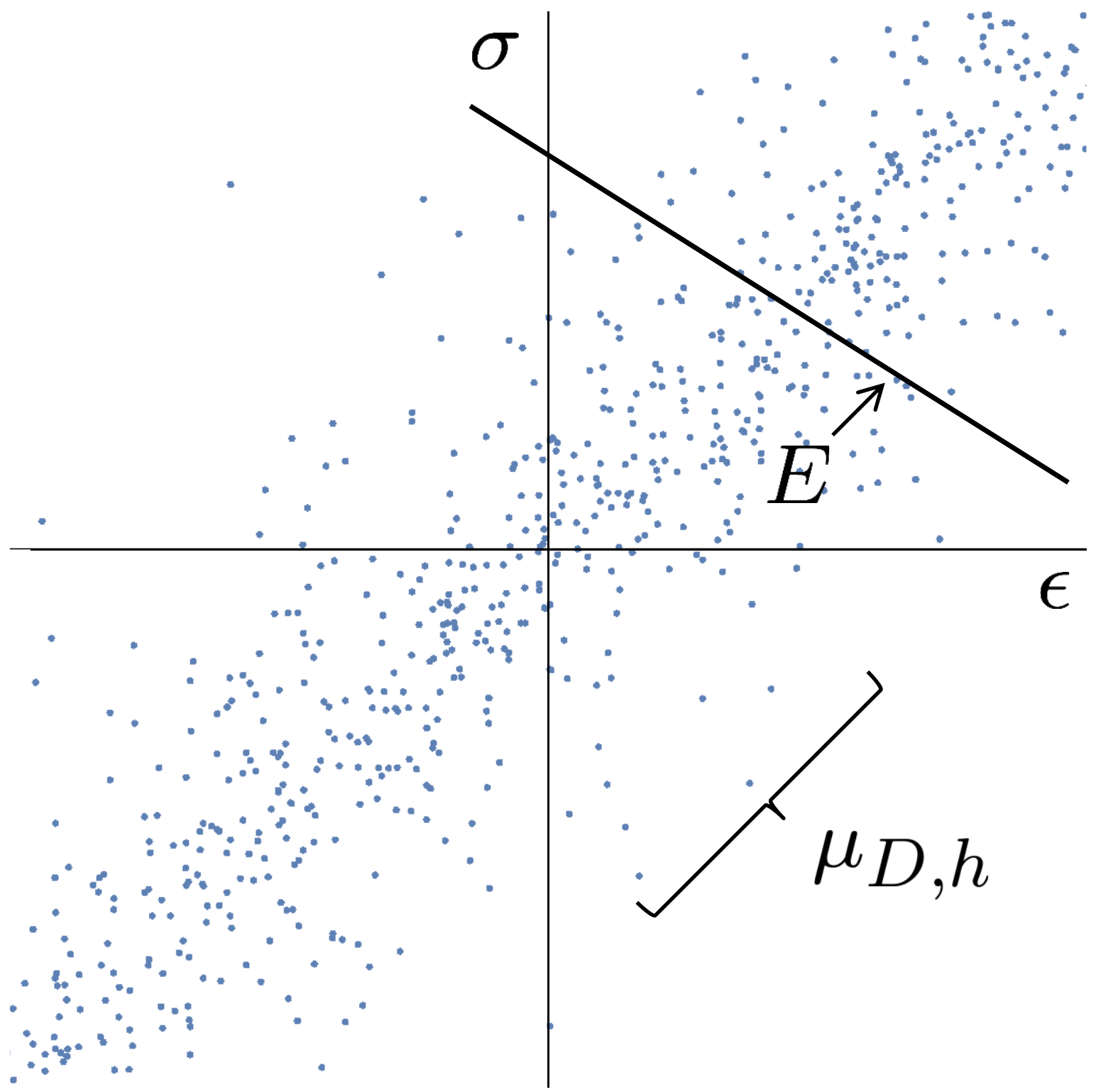}
	\end{subfigure}
    \caption{Stochastic case. a) Material likelihood function $L_D$ in the form of a sliding Gaussian (dark: low likelihood; light: high likelihood), constraint set $E$ and likelihood function $L$ obtained by restricting $L_D$ to $E$. b) Empirical likelihood measure $\mu_{D,h}$ sampled from $L_D$. Note empty intersection of the support of $\mu_{D,h}$ with $E$.} \label{lRkk0p}
\end{center}
\end{figure}

\subsection{Problem Set-Up}\label{sec:setup}
In the present work, we extend this deterministic framework to stochastic systems. To this end, we assume that material behavior is characterized by a Radon measure $\mu_D$, or {\sl material measure}, defined over the phase space $Z$, with the property that $\mu_D(A)$ measures the likelihood of observing in the laboratory a material state in the set $A \subset Z$. In this manner, we allow the behavior of the material to be intrinsically stochastic, i.~e., the spread of the measure $\mu_D$ is not necessarily the result of measurement error, or experimental scatter, but the result of randomness of the material behavior itself. A fundamental difficulty inherent to such a representation is that the material measure $\mu_D$ is not finite in general and, in particular, it cannot be normalized to define a probability measure. For instance, consider a deterministic elastic material characterized by a local material law $\sigma = \hat{\sigma}(\epsilon)$, with $(\epsilon,\sigma) \in \mathbb{R}^{d\times d}_{\rm sym} \times \mathbb{R}^{d\times d}_{\rm sym} \equiv Z_{\rm loc}$, and where $\hat{\sigma}(\cdot)$ is a locally Lipschitz material law describing the behavior of one material point. Then the corresponding material measure is $\mu_{D,{\rm loc}} = \mathcal{H}^{d(d+1)/2} \LL D_{\rm loc}$, where $D_{\rm loc} = \{(\epsilon,\sigma) \, : \, \sigma = \hat{\sigma}(\epsilon)\}$ is the graph of the local material law. Evidently, $\mu_{D,{\rm loc}}$ is a Radon measure over $Z_{\rm loc}$ but it is not finite.

For the sake of generality and without significant additional complexity, we also allow the loading to be random, and we describe the field constraints by means of a second Radon measure $\mu_E$ over $Z$, or {\sl constraint measure}, with the property that $\mu_E(A)$ measures the likelihood of finding an admissible state in $A \subset Z$. For instance, for an elastic material, $\mu_E(A)$ returns the likelihood of finding a pair $z = (\epsilon, \sigma)$ in $A$ with $\epsilon$ compatible and $\sigma$ in equilibrium with the random loading. As already noted, in the particular case of deterministic loading the field equations restrict admissible states to an affine subspace $E$ of $Z$. For $Z = \mathbb{R}^N \times \mathbb{R}^N$ finite dimensional, we show in Section~\ref{RueZ3N} that in many examples the structure of the field equations implies that the admissible space $E$ is indeed a subspace of $Z$ of dimension $N$ and co-dimension $N$. In this case, the corresponding constraint measure is, therefore, $\mu_E = \mathcal{H}^N \LL E$. Again, we note that $\mu_E$ is a Radon measure over $Z$ but it is not finite.

\subsection{Our Contributions}\label{sec:contributions}

For a system thus defined, the classical inference problem consists of determining the likelihood of observing a material state $y \in Z$ and an admissible state $z \in Z$ conditioned by the requirement that $y=z$. 
For suitable choices of those measures, we introduce a likelihood measure which can be interpreted as the intersection $\mu_D\cap\mu_E$ of $\mu_D$ and $\mu_E$.
By way of example, if $Z = \mathbb{R}^N \times \mathbb{R}^N$, $\mu_D = L_D \mathcal{L}^{2N}$ and $\mu_E = L_E \mathcal{L}^{2N}$, for some continuous material and constraint likelihood functions $L_D$ and $L_E$, respectively, then $\mu_D \cap \mu_E = L_D L_E \mathcal{L}^{2N}$. If $Z = \mathbb{R}^N \times \mathbb{R}^N$, $\mu_D = L_D \mathcal{L}^{2N}$ with continuous material likelihood function $L_D$ and $\mu_E = \mathcal{H}^N \LL E$, corresponding to deterministic loading, then $\mu_D \cap \mu_E = L_D \mathcal{H}^N \LL E$, Fig.~\ref{lRkk0p}a. We note that, whereas neither $\mu_D$ nor $\mu_E$ are finite in general, we expect the intersection $\mu_D \cap \mu_E$ to be finite and non-degenerate in the cases of interest, i.~e., $0 < |\mu_D \cap \mu_E| < +\infty$. In particular, $\mu_D \cap \mu_E$ can then be normalized by $|\mu_D \cap \mu_E|$ to define a probability measure characterizing the expectation of outcomes of the system. The condition that the intersection $\mu_D \cap \mu_E$ be well-defined, finite and non-degenerate sets forth a general notion of {\sl transversality} between the measures $\mu_D$ and $\mu_E$.

For definiteness, we restrict attention to finite-dimensional systems, $Z = \mathbb{R}^N \times \mathbb{R}^N$
and introduce in Definition~\ref{def:diag} 
and Remark~\ref{def:diag2} new concepts of {\sl diagonal concentration}, {\sl intersection of measures} and {\sl transversality}.
Specifically, we consider the product measure $\mu = \mu_D \times \mu_E$ over $Z \times Z$ and penalize deviations from the diagonal ${\rm diag}(Z\times Z)$ by means of parameterized Gaussian weights $w_\beta$ converging weakly to $\mathcal{H}^N \LL {\rm diag}(Z\times Z)$. We further assume
transversality, in the sense that the weighted, or {\sl thermalized}, measures $\mu_\beta = w_\beta \mu$ converge weakly to a measure $\mu_\infty$ as $\beta \to +\infty$ (this is different from the usage of the term transversality in \cite{conti2018data}). It then follows, Lemma~\ref{lemmadiagconc}, that the measure $\mu_\infty$, referred to as the {\sl diagonal concentration} of $\mu$, is diagonal, i.~e., it is supported on ${\rm diag}(Z\times Z)$. For measures $\mu = \mu_D \times \mu_E$ for which this procedure is well-defined, the diagonal concentration measure  $\mu_\infty$ supplies a convenient representation of the intersection measure $\mu_D \cap \mu_E$ and, by extension, of the solution of the classical inference problem. In Section~\ref{tw7QP8}, we show using the Kullback-Leibler divergence that thermalization may be regarded as an entropic regularization and the diagonal concentration as the corresponding athermal limit, which lends motivation to the choice of terminology.

In Sections~\ref{LlWZ3G} and \ref{nXr8mw}, we present Theorems~\ref{thm:concentration-Lebesgue} and \ref{thm:det-Leb} that illustrate two cases in which the thermalization approach to diagonal concentration is well-defined. Both theorems are concerned with finite-dimensional systems with phase space $Z = \mathbb{R}^N \times \mathbb{R}^N$. The first case is concerned with joint likelihood measures $\mu = L \, \mathcal{L}^{2N} \times \mathcal{L}^{2N}$ that are absolutely continuous with respect to the Lebesgue measure over $Z\times Z$ with regular density $L$. This scenario allows for general correlations between the likelihoods of material and admissible states. In this case, Theorem~\ref{thm:concentration-Lebesgue} sets forth regularity, continuity and equi-integrability conditions for $L$ ensuring the existence of the diagonal concentration $\mu_\infty$. The second scenario is concerned with the deterministic loading case, $\mu_E = \mathcal{H}^N \LL E$, and material measures $\mu_D = L_D \, \mathcal{L}^{2N}$ that are absolutely continuous with respect to the Lebesgue measure over $Z$ with regular density $L_D$. Here again, Theorem~\ref{thm:det-Leb} supplies regularity, continuity and equi-integrability conditions on $L_D$ and $E$ ensuring the transversality of $\mu_D$ and $\mu_E$ and the existence of the diagonal concentration $\mu_\infty$, in the sense of Definition~\ref{def:diag}. Examples are also presented that illustrate the scope of the theorems.

Beyond characterizing the intersection of transverse measures in a convenient manner, thermalization proves crucial in cases in which the material measure $\mu_D$ is known only approximately through a sequence of approximating discrete measures $(\mu_{D,h})$, e.~g., corresponding to experimental measurements, converging weakly to $\mu_D$, Fig.~\ref{lRkk0p}b. In this case, the intersection measures $\mu_{D,h} \cap \mu_E$, even assuming that it can be defined, may be degenerate, e.~g., if the loading is deterministic with no intersection between the support of $\mu_{D,h}$ and the constraint set $E$. Under these conditions, the sequence $(\mu_{D,h} \cap \mu_E)$ fails to approximate the limiting intersection $\mu_D \cap \mu_E$ in any meaningful way. More generally, if the system is subject to random loading and $\mu$ is approximated by a sequence $(\mu_h)$ of discrete measures converging weakly to $\mu$, the diagonal concentrations $(\mu_{h,\infty})$ are almost surely degenerate and fail to characterize $\mu_\infty$ in the limit.

To circumvent this difficulty, we again resort to thermalization and define an auxiliary sequence of thermalized empirical measures $(\mu_{h,\beta_h})$ with the aid of a carefully chosen sequence $\beta_h \to +\infty$. Here, the effect of thermalization is to regularize the discrete measures $\mu_h$ so that the athermal limit $\mu_{h,\beta_h}\rightarrow \mu_\infty$ is well-defined for some limiting measure $\mu_\infty$. In the case of deterministic loading, we rely on a suitable thermalization $\mu_{D,h,\beta_h}$ of the sequence $(\mu_{D,h})$ of empirical material sets to intersect properly with $\mu_E$ and deliver a well-defined limiting measure $\mu_\infty$. Evidently, the central question then concerns whether the limiting measure $\mu_\infty$ delivered by $(\mu_{h,\beta_h})$ is indeed the diagonal concentration of $\mu$. Intuitively, we anticipate a need for the cooling sequence $\beta_h$ to be slow enough, i.~e., to define a {\sl quenching} sequence, lest the regularizing effect of thermalization be lost along the sequence.

In Sections~\ref{iw71LU} and \ref{lH1P1g}, we present Theorems~\ref{thm:randload} and \ref{tQ4UnZ} that establish conditions under which the quenching procedure just described is well-defined and properly convergent. The theorems are again restricted to finite-dimensional systems with phase space $Z = \mathbb{R}^N \times \mathbb{R}^N$ and concerned with the same scenarios considered in Section~\ref{DOsT8G}, namely, Lebesgue absolutely-continuous likelihood measures $\mu = L \, \mathcal{L}^{2N} \times \mathcal{L}^{2N}$ and Lebesgue absolutely-continuous materials measures $\mu_D = L_D \, \mathcal{L}^{2N}$ in systems under deterministic loading, $\mu_E = \mathcal{H}^N \LL E$. The theorems set forth sufficient conditions on the convergence $\mu_h \rightharpoonup \mu$, conversely $\mu_{D,h} \rightharpoonup \mu_D$, and the quenching sequence $\beta_h$ ensuring the convergence of the thermalized sequence $(\mu_{h,\beta_h})$ to the diagonal concentration $\mu_\infty$ of $\mu$. Again, we illustrate the requirements and scope of the theorems by means of examples.

Chief among these, from the standpoint of applications, is the case of approximation by means of discrete measures $(\mu_h)$, or $(\mu_{D,h})$ as the case may be, presented in Sections~\ref{yu04L1} and \ref{8yK9TU}. We show that, in these cases, the thermalization and quenching procedure results in analytical expressions for the expectation $\mathbb{E}[f]$ of bounded continuous functions $f$ that are explicit in the data and require no intervening modeling for their computation. These results demonstrate the feasibility of the model-free data-driven inference paradigm and, in particular, the possibility of making convergent inferences on system outcomes directly from data. These observations notwithstanding, it should be carefully noted that the sums involved in the explicit analytical expressions for the expectations are of combinatorial complexity in the dimension of phase space and the size of the material set. Sums of this type can be effectively implemented and computed by means of Monte Carlo methods, for which there is an extensive literature in computational physics. However, these aspects of implementation require careful attention and are beyond the scope of the present paper.

\section{Deterministic Data-Driven mechanics in finite dimensions}
\label{sec:det}

For completeness and by way of introduction, we begin with a brief summary of deterministic Data-Driven (DD) mechanics of finite-dimensional systems as introduced in \cite{kirchdoerfer2016data,kirchdoerfer2017data}.

\subsection{Phase space, compatibility and conservation constraints}\label{RueZ3N}

The field theories of science have a particular structure that pervades disparate fields of application 
and which set them apart from other data-intensive fields. Here, we restrict attention to finite-dimensional systems comprising $m$ components whose state is characterized by two work-conjugate fields $\epsilon \equiv \{\epsilon_e \in \mathbb{R}^d,\ e=1,\dots,m\}$ and $\sigma \equiv \{\sigma_e \in \mathbb{R}^d,\ e=1,\dots,m\}$. We refer to the space of pairs $Z_e = \{z_e \equiv (\epsilon_e, \sigma_e) \in \mathbb{R}^d \times \mathbb{R}^d\}$ as the {\sl local phase space} of the component $e$, and $Z = Z_1 \times \cdots \times Z_m = \mathbb{R}^{2N}$, $N = m d$, as the {\sl global phase space} of the system.
In the entire paper we assume that $Z$ is endowed with a scalar product and denote by $\|\cdot\|$ the corresponding norm; for $(y,z)\in Z\times Z$ we use $\vertiii{(y,z)}:=\sqrt{\|y\|^2+\|z\|^2}$. We denote by $|x|_n$ the Euclidean norm 
of $x\in\mathbb R^n$ and by $x\cdot y$ the corresponding scalar product.

\begin{example}[Trusses]\label{nBpq7d}{\rm
We illustrate the essential structure of discrete field theories by means of the simple example of truss structures. Trusses are assemblies of bars that deform in uniaxial tension or compression. The bars are articulated at common joints, or nodes, that act as hinges, i.~e., cannot transmit moments. Trusses are examples of connected networks that obey conservation laws. Other examples in the same class include electrical circuits, pipeline networks, traffic networks, and others.

The material behavior of a bar $e$ is characterized by a particularly simple relation between uniaxial strain $\epsilon_e$ and uniaxial stress $\sigma_e$. Thus, in this case $d=1$ and  the local phase spaces are $Z_e = \mathbb{R} \times \mathbb{R}$. These local states are subject to the following laws:

i) Compatibility: Suppose that bar $e$ is connected to nodes $a$ and $b$. Then, the strain in the bar is
\begin{equation}
    \epsilon_e = \frac{u_b - u_a}{L_e} \cdot d_e ,
\end{equation}
where $L_e$ is the length of the bar $e$, $d_e \in \mathbb{R}^3$ is the unit vector pointing from $a$ to $b$ and $u_a$, $u_b \in \mathbb{R}^3$ are the displacements of the $a$ and $b$, respectively.

ii) Equilibrium: Let $S_a$ be the star of an unconstrained node $a$, i.~e., the collection of bars connected to $a$. Then, we must have
\begin{equation}
    \sum_{e\in S_a} \sigma_e d_e A_e + f_a = 0,
\end{equation}
where $A_e$ is the cross-sectional area of bar $e$, $d_e$ points from $a$ to the node connected to $a$ by bar $e$, and $f_a\in\mathbb{R}^3$ is the force applied to node $a$.} \hfill$\square$
\end{example}

As the above example indicates, in many cases the state of the system is subject to linear constraints of the general form
\begin{subequations}\label{9qHWzU}
\begin{align}
    &
    \sum_{e=1}^m w_e B_e^T \sigma_e = f ,
    \\ &
    \epsilon_e = B_e u + g_e , \quad e = 1,\dots m ,
\end{align}
\end{subequations}
where $u \in \mathbb{R}^{n}$ is the array of degrees of freedom of the system, $w_e$ are positive weights, $B_e \in \mathbb{R}^{d \times n}$ is a discrete gradient operator, $B_e^T$ is a discrete divergence operator, $f \in \mathbb{R}^n$ is a force array resulting from distributed sources and Neumann boundary conditions and the arrays $g_e \in \mathbb{R}^{d}$ follow from Dirichlet boundary conditions. The constraints (\ref{9qHWzU}) are material independent and define an affine subspace $E$ of $Z$, the {\sl constraint set}. The constraint set $E$ encodes all the data of the problem, including geometry, loading and boundary conditions. The constraints (\ref{9qHWzU}) can also be expressed in matrix form as
\begin{subequations}\label{vZftZT}
\begin{align}
    &
    B^T \tau = f
    \\ &
    \epsilon = B u + g ,
\end{align}
\end{subequations}
with $B = (B_1, \dots, B_m) \in \mathbb{R}^{N \times n}$, $\epsilon = (\epsilon_1, \dots, \epsilon_m)  \in \mathbb{R}^N$, $\tau = (w_1 \sigma_1,\dots,$ $w_m \sigma_m)$  $\in$ $\mathbb{R}^N$ and $g = (g_1, \dots, g_m)  \in \mathbb{R}^N$.

We note that the affine space $E$ defined by the constraints (\ref{vZftZT}) is a translate of the linear space $E_0$ defined by the homogeneous constraints
\begin{subequations}
\begin{align}
    &   \label{qmS3Q5}
    B^T \tau = 0
    \\ &  \label{jJ9B6T}
    \epsilon = B u .
\end{align}
\end{subequations}
Evidently, $E_0 = E_\epsilon \times E_\sigma$, where $E_\epsilon$ is the linear space defined by (\ref{jJ9B6T}) and $E_\sigma$ is the linear space defined by (\ref{qmS3Q5}). Therefore, we have
\begin{equation}
    {\rm dim}(E_0)
    =
    {\rm dim}(E_\epsilon)
    +
    {\rm dim}(E_\sigma)
    =
    {\rm dim}({\rm Im}(B)) + {\rm dim}({\rm Ker}(B^T))
    =
    N .
\end{equation}
Since $Z = \mathbb{R}^N \times \mathbb{R}^N$, it follows that the constraint set $E$ is an affine subspace of $Z$ of dimension $N$ and co-dimension $N$.

This observation  characterizes the structure and dimensionality of the phase space $Z$ and of the subspace $E$ of all admissible states in $Z$, or constraint set, i.~e., the set of all the states that satisfy the conservation laws (\ref{vZftZT}). 

\subsection{Material characterization}

We assume that the behavior of the material of each component $e=1,\dots,m$ of the system is characterized by---possibly different---local material data sets $D_e$ of pairs $z_e \equiv (\epsilon_e, \sigma_e)$, or {\sl local states}. For instance, each point in the data set may correspond to, e.~g., an experimental measurement, a subgrid multiscale calculation, or some other means of characterizing material behavior. The local material data sets can be point sets, graphs or sets of any arbitrary dimension. The set $D = D_1 \times \cdots \times D_m$ is the global material data set.

\begin{example}[Material laws]\label{NAcwZE}{\rm
Classically, material behavior is often characterized by a convex function $W_e : \mathbb{R}^d \to \mathbb{R}$, with dual $W_e^*$ such that the $d$-dimensional graph
\begin{equation}\label{NVajuA}
    D_e = \{ y_e = (\epsilon_e, \sigma_e) \in \mathbb{R}^d \times \mathbb{R}^d \, : \,
    \sigma_e = DW_e(\epsilon_e) \}
\end{equation}
is the local material set. For instance, for linear material behavior, Fig.~\ref{3rlASw}a, we have
\begin{equation}\label{b89vda}
    W_e(\epsilon_e)
    =
    \frac{1}{2}
    \mathbb{C}_e \epsilon_e \cdot \epsilon_e ,
    \qquad
    W_e^*(\sigma_e)
    =
    \frac{1}{2}
    \mathbb{C}_e^{-1} \sigma_e \cdot \sigma_e ,
\end{equation}
where $\mathbb{C}_e>0$ is a fixed scalar.
}
\hfill$\square$
\end{example}

\begin{example}[Point data sets]\label{AEZUw9}{\rm
Another common situation concerns materials that are characterized experimentally and whose behavior is known only through a point data set $D_e$, Fig.~\ref{3rlASw}b, collecting the results of experimental tests. Evidently, such material data sets do not define graphs, much less affine subspaces, in the local phase space $Z_e$.} \hfill$\square$
\end{example}

\subsection{Classical solutions}\label{JzZtK7}

Consider now a system with phase space $Z = \mathbb{R}^N \times \mathbb{R}^N$ characterized by a material data set $D$ and a constraint set in the form of an affine space $E$ of dimension $N$ and codimension $N$. Evidently, the set of {\sl classical solutions} of the system is the intersection $D \cap E$, i.~e., material states of the system that are admissible in the sense of satisfying all equilibrium and all compatibility constraints.

\subsection{Data-Driven (DD) reformulation}

As noted in \cite{kirchdoerfer2016data}, the notion of classical solution is too rigid in cases where solutions may be reasonably expected to exist but for which $D \cap E = \emptyset$ due to the paucity of the data, e.~g., taking the form of a point set, see Example~\ref{AEZUw9}. The analysis of such systems requires a suitable extension of the notion of solution.

Data-Driven (DD) solvers seek to determine the material state $y \in D$ that is closest to being admissible, in the sense of $E$, or, alternatively, the admissible state $z \in E$ that is closest to being a possible state of the material, in the sense of $D$. Optimality is understood in the sense of a suitable norm, e.~g., for the set-up in Example~\ref{NAcwZE}, we may choose
\begin{equation}\label{Poyet2}
    \| z \|
    =
    \left(
    \sum_{e=1}^m
        w_e
        \Big(
            \mathbb{C}_e |\epsilon_e|_d^2
            +
            \mathbb{C}_e^{-1} |\sigma_e|_d^2
        \Big)
    \right)^{1/2} ,
\end{equation}
with $w_e > 0$ as in \eqref{9qHWzU}, $\mathbb{C}_e > 0$ as in \eqref{b89vda}, $e=1,\dots,m$.  The corresponding DD problem is, then,
\begin{equation}\label{sTO0ic}
    \inf_{y \in D} \inf_{z \in E} \| y - z \|^2
    =
    \inf_{z \in E} \inf_{y \in D} \| y - z \|^2 ,
\end{equation}
i.~e., we wish to determine the state $z \in E$ of the system that is admissible and closest to the data set $D$, or, equivalently, the point $y \in D$  in the material data set that is closest to being admissible.

Evidently, if $E$ is affine and $D$ is compact, e.~g., consisting of a finite collection of points, then the DD problem (\ref{sTO0ic}) has solutions by the Weierstrass extreme-value theorem. More generally, in \cite[Cor.~2.9]{conti2018data} the following is proved.

\begin{proposition}[Existence of DD solutions]\label{1Lublu}
Let $Z = \mathbb{R}^N \times \mathbb{R}^N$, $E$ an affine subspace of $Z$ and
$D$ a non-empty closed subset of $Z$. Suppose that the condition
\begin{equation}\label{FRaf5E}
    \| y - z \| \geq c ( \| y \| + \| z \| ) - b
\end{equation}
holds for all $y \in D$, $z \in E$, with some constants $c > 0$ and $b \geq 0$. Then, the DD problem (\ref{sTO0ic}) has at least one solution.
\end{proposition}

\begin{proof}
Let $(y_h, z_h) \subset D \times E$ be a minimizing sequence. By
\eqref{FRaf5E}, $(y_h)$ and $(z_h)$ are bounded in $Z$. Passing
to subsequences, there is $(y,z) \in {Z\times Z}$ such that $(y_h,z_h) \to
(y,z)$. By the closedness of $E$ and $D$, it follows that $(y,z) \in D \times
E$. By the continuity of the norm,
\begin{equation}
\begin{split}
    &\inf_{(y',z') \in D \times E} \| y' - z' \|^2
    \leq
    \| y - z \|^2
    \\ &\qquad =
    \lim_{h\to\infty}
    \| y_h - z_h \|^2
    =
    \inf_{(y',z') \in D \times E} \| y' - z' \|^2 ,
\end{split}
\end{equation}
and $(y,z)$ is a DD solution.
\end{proof}

We note that condition (\ref{FRaf5E}) fails when the distance between $D$ and $E$ is minimized at infinity, in which case the minimizing sequences diverge and solutions fail to exist. This condition (which was called transversality in \cite{conti2018data}) is related to, but different from, the transversality condition that plays an important role in the probabilistic extension of DD, as evinced in Section~\ref{sec:concentration}.

\begin{example}[Linear trusses] {\rm A simple example is furnished by a linear-elastic response of the form $D_e = \{ \sigma_e = \mathbb{C}_e \epsilon_e \}$, where $\mathbb{C}_e$ are elastic moduli, and linear equilibrium and compatibility constraints of the form (\ref{9qHWzU}). This set-up is combining Examples~\ref{nBpq7d} and \ref{NAcwZE}.
A straightforward calculation \cite{kirchdoerfer2016data} then shows that
\begin{equation}\label{dI3roc}
    d^2(z,D)
    = \inf_{y\in D} \|y-z\|^2
    =
    \sum_{e=1}^m
        \frac{1}{4}
        w_e
        \mathbb{C}_e^{-1}
        (\sigma_e - \mathbb{C}_e \epsilon_e)
        \cdot
        (\sigma_e - \mathbb{C}_e \epsilon_e)
\end{equation}
and that the transversality condition (\ref{FRaf5E}) is equivalent to the condition that $B^T \mathbb{C} B$ $>$ $0$ with $\mathbb{C} = {\rm diag}(\mathbb{C}_1,\dots, \mathbb{C}_m)$, i.~e., if the stiffness matrix of the system is strictly positive definite. Under these conditions, Theorem~\ref{1Lublu} ensures existence. In addition, by the strict convexity of (\ref{dI3roc}) the solution is unique.} \hfill$\square$
\end{example}

\section{Inference by diagonal concentration}\label{sec:concentration}

We recall from the previous discussion that the states of the systems of interest can be identified with points in a phase space $Z = \mathbb{R}^N \times \mathbb{R}^N$ normed by some convenient norm such as (\ref{Poyet2}). In addition, the geometry, loading, equilibrium and compatibility constraints acting on the system define an affine subspace $E$ of $Z$, or constraint set, of dimension $N$ and co-dimension $N$. In the preceding deterministic formulation of the DD problem, the possible states of the material are additionally known to be in a material data set $D$ for sure. We wish to extend this theory to cases in which the material behavior, and possibly the loading, are inherently random and defined in probabilistic terms.

For a metric space $X$ we denote by $\mathcal M(X)$ the set of Radon measures on $X$. Assume that the behavior of the material is random and characterized by a positive  {\sl material likelihood measure} $\mu_D \in \mathcal{M}(Z)$ representing the {\sl likelihood} of observing a material state in the laboratory. We note that likelihood measures are not necessarily probability measures, indeed need not be finite, and are defined modulo positive multiplicative constants. For instance, if $\mu_D$ is absolutely continuous with respect to the Lebesgue measure, then
\begin{equation}
    \mu_D = L_D \mathcal{L}^{2N},
\end{equation}
where $L_D:Z\to[0,\infty)$ is a {\sl material likelihood function}. The function $\Phi_D:Z\to (-\infty,\infty]$,
\begin{equation}\label{gPUn6S}
    \Phi_D(y) := - \log L_D(y)
\end{equation}
is the corresponding {\sl material potential}. The measure $\mu_D$ can be understood as representing the likelihood of observing the material in a certain state, and can in principle be approximated by performing a large set of measurements and considering, on a suitable scale, the density of data points. Thus, regions of phase space that are sparsely covered by data are less likely to be observed than densely covered regions. Material likelihood measures that are singular with respect to the Lebesgue measure are also of interest. For instance, discrete empirical measures of the form
\begin{equation}\label{7KTLs5}
    \mu_D = \sum_{i = 1}^\infty c_i \delta_{y_i} ,
\end{equation}
play a central role in approximation. The deterministic case corresponds to the case $\mu_D = \mathcal{H}^N \LL D$, where the data set $D$ is a graph in $Z$ of dimension $N$, representing the material law. For example, one could have $D=D_1\times\cdots\times D_m$, with $D_e$ as in \eqref{NVajuA} or a corresponding nonlinear generalization.

\begin{example}[Local material behavior] \label{dzeXXX}{\rm
Material behavior is often {\sl local} and can be characterized over each local phase space $Z_e = \mathbb{R}^d \times \mathbb{R}^d$ by a local material measure $\mu_{D,e} \in \mathcal{M}(Z_e)$. Assuming that the behavior of the members is independent, the {\sl global} material measure is then given by the product measure
\begin{equation}
    \mu_D = \mu_{D,1} \times \cdots \times \mu_{D,m} .
\end{equation}
If the local material measures $\mu_{D,e}$ are defined in terms of {\sl local} likelihood functions $L_{D,e}$ and {\sl local} potentials $\Phi_{D,e}$, then the {\sl global} likelihood function is
\begin{equation}
    L_D = L_{D,1} L_{D,2}\cdots L_{D,m-1}L_{D,m} ,
\end{equation}
i.~e., the product of the local likelihood functions of the members, and the {\sl global} potential is
\begin{equation}
    \Phi_D = \Phi_{D,1} + \cdots + \Phi_{D,m} ,
\end{equation}
i.~e., the sum of the local potentials of the members. \hfill$\square$}
\end{example}

Without much additional complexity, we may assume that the boundary conditions and loading are also random and characterized by a Radon {\sl constraint likelihood measure} $\mu_E \in \mathcal{M}(Z)$ representing the relative likelihood or different boundary conditions and loading acting on the system. The case of deterministic loading is recovered by setting $\mu_E = \mathcal{H}^N \LL E$, with $E\subseteq Z$ an affine subspace of dimension $N$. More generally, we may suppose that the material behavior and the constraints on the system are jointly characterized by a likelihood measure $\mu \in \mathcal{M}(Z\times Z)$. If material behavior and constraints are uncorrelated, the likelihood measure is the product measure
\begin{equation}\label{VahD8C}
    \mu = \mu_D \times \mu_E .
\end{equation}
The joint likelihood function $\mu$ thus defined fully characterizes the system, both as regards material behavior and constraints.

Within this framework, the classical inference problem is to determine the likelihood of observing a material state $y \in Z$ and an admissible state $z \in Z$ conditioned to $y=z$, i.~e., conditioned to the material and admissible states being equal. We expect the likelihood measure $\mu_\infty$ of such pairs of states to be a certain concentration of $\mu$ to the diagonal ${\rm diag}(Z \times Z)$ (for the precise notion, see Definition~\ref{def:diag}). For $\mu\in\mathcal{M}(X)$, we define by
\begin{equation}
    |\mu|:=\mu(X) \in [0,\infty]
\end{equation}
the \emph{total variation} of $\mu$.
If $\mu_\infty \in \mathcal{M}(Z \times Z)$ is finite and non-degenerate, i.~e.,
\begin{equation}
    0 < | \mu_\infty | < +\infty ,
\end{equation}
then $\mu_\infty$ can be normalized to define the probability measure
\begin{equation}\label{fXM3k2}
    \nu_\infty := \frac{\mu_\infty}{| \mu_\infty |} ,
\end{equation}
which gives the probability of material and admissible states of the system and the expectation
\begin{equation}
    \mathbb{E}_\infty(f) = \int_{Z\times Z} f(y,z) \, d\nu_\infty(y,z)
\end{equation}
of bounded continuous functions defined on the extended phase space, $f \in C_b(Z\times Z)$.  We shall show below that $\nu_\infty$ concentrates on the diagonal, so that univariate functions $f\in C_b(Z)$ are sufficient to fully characterize it.
The probability measure $\nu_\infty$, eq.~(\ref{fXM3k2}), may be regarded as the solution of the classical inference problem defined by $\mu$. The marginals
\begin{equation}\label{ZZp6MH}
    \pi_D \nu_\infty
    =
    \pi_E \nu_\infty ,
\end{equation}
are the corresponding probability measures of material states and system outcomes, respectively.

\begin{example}[Random trusses]{\rm
Consider a truss such as defined in Example~\ref{nBpq7d}, where $d=1$, $w_e=1$, and with material behavior characterized by the local likelihood functions
\begin{equation}\label{nJFGU9}
    L_{D,e}(\epsilon_e,\sigma_e)
    =
    \exp
    \Big(
        -\frac{1}{2}
        \mathbb{C}_e^{-1}
        (\sigma_e-\mathbb{C}_e\epsilon_e)
        \cdot
        (\sigma_e-\mathbb{C}_e\epsilon_e)
    \Big) .
\end{equation}
We begin by parameterizing the constraint space $E$. Recall that $m$ is the number of members of the truss, $Z = \mathbb{R}^m \times \mathbb{R}^m$ is the phase space and we assume $n < m$ the number of unconstrained degrees of freedom. Let $l = m-n$ and $A \in \mathbb{R}^{m\times l}$ the matrix whose columns define a basis of ${\rm Ker}(B^T)$. Then $\sigma$ satisfies the equilibrium condition $B^T\sigma = 0$ if and only if there is $v \in \mathbb{R}^l$ such that
\begin{equation}\label{RYaYcU}
    \sigma = A v .
\end{equation}
We may thus regard $v$ as a discrete Airy potential and $A$ as a discrete Airy operator. In addition, suppose that $\epsilon$ satisfies the compatibility condition $A^T \epsilon = 0$ if and only if there are displacements $u$ such that $\epsilon = B u$. Then, the constraint set $E_0$ through the origin, corresponding to $f=0$ and $g=0$, admits the representation
\begin{equation}
    E_0
    =
    \{
    (\epsilon,\sigma) \in Z = \mathbb{R}^m \times \mathbb{R}^m
    \, : \,
    \epsilon = B u,\ u \in \mathbb{R}^n;
    \ \sigma = A v,\ v \in \mathbb{R}^l \} .
\end{equation}
The general constraint set is then the translation
\begin{equation}
    E = z_0 + E_0,
    \quad
    z_0 = (\epsilon_0,\sigma_0) ,
    \quad
    \epsilon_0 = g,
    \quad
    B^T \sigma_0 = f  .
\end{equation}
If $\mathop\mathrm{rank} B=n$ then the dimension of $Z$ is $2N = 2m$ and the dimension of $E$ is $N = m = n + l$. We may regard $(u,v) \in \mathbb{R}^n \times \mathbb{R}^l$ as a set of coordinates parameterizing $E$. In this representation, the likelihood function of outcomes takes the form
\begin{align}
    L(u,v)
    &:= \prod_{e=1}^m L_{D,e}\left((\epsilon_0+Bu)_e, (\sigma_0+Av)_e\right)\\
        &=
    \exp
    \Big(
        -
        \frac{1}{2}
         \mathbb{C}^{-1} q\cdot q
    \Big) , 
    \quad q:=Av+\sigma_0-\mathbb{C}(Bu+\epsilon_0)\in \mathbb{R}^m
\end{align}
where we write $\mathbb{C} := {\rm diag}(\mathbb{C}_1, \cdots, \mathbb{C}_m)$.
We assume that $Z$ is the direct sum of $A\mathbb R^l$ and $\mathbb C B \mathbb R^n$. Since $B^TA=0$, from the properties of Gaussian integrals, the corresponding normalizing factor is computed as
\begin{equation}
\int_{E}\prod_{e=1}^m L_{D,e}\left(\epsilon_e, \sigma_e\right)d\mathcal H^N(\epsilon,\sigma)
 =\sqrt{\det(A^TA)\det(B^TB)} \int_{\mathbb R^n\times\mathbb R^l} L(u,v) dudv 
\end{equation}
with
\begin{equation}
\int_{\mathbb R^m} L(u,v) \, du \, dv
    =
    \frac{1}{\sqrt{\det(B^T \mathbb{C} B /2\pi )}}
    \frac{1}{\sqrt{\det(A^T \mathbb{C}^{-1} A /2\pi )}} .
\end{equation}
Again, from the properties of Gaussian integrals, the expected values $(\bar{u},\bar{v})$ of the coordinates $(u,v)$ are obtained as the unique solution of $A\bar v+\sigma_0-\mathbb C B\bar u - \mathbb C \epsilon_0 =0$ or minimizing the potential $\Phi=-\log(L)$, with the result,
\begin{subequations}
\begin{align}
    &
    \bar{u} = (B^T\mathbb{C}B)^{-1} B^T(\sigma_0 - \mathbb{C} \epsilon_0) ,
    \\ &
    \bar{v} = (A^T\mathbb{C}^{-1}A)^{-1} A^T(\epsilon_0 - \mathbb{C}^{-1} \sigma_0) ,
\end{align}
\end{subequations}
which are computed by inverting the stiffness and compliance matrices of the truss, $B^T \mathbb{C} B$ and $A^T\mathbb{C}^{-1}A$, respectively. The probability density of outcomes then follows as
\begin{equation}
\begin{split}
    &
    \rho(u,v)
    =
    \sqrt{\det(B^T \mathbb{C} B /2\pi )}
    \sqrt{\det(A^T \mathbb{C}^{-1} A /2\pi )}
    \times \\ &
    \exp
    \Big(
        -
        \frac{1}{2}
        (B^T\mathbb{C}B)(u-\bar{u}) \cdot (u-\bar{u})
        -
        \frac{1}{2}
        (A^T\mathbb{C}^{-1}A)(v-\bar{v}) \cdot (v-\bar{v})
    \Big) ,
\end{split}
\end{equation}
which provides a full account of the probability of outcomes.
\hfill$\square$}
\end{example}

\section{Thermalization}\label{DOsT8G}

The concentration operation just described is a non-trivial proposition for general measures. For instance, in the case $\mu = \mu_D \times \mu_E$, the problem of finding the diagonal concentration of $\mu$ is related to the problem of intersection of measures \cite{Federer:1969,mattila1984hausdorff}. 
Here, for definiteness, we opt for characterizing diagonal concentration by means of a limiting procedure based on thermalization and restrict attention to cases in which this procedure is successful. In particular, we restrict to the following three cases: (1) measures $\mu$ that are absolutely continuous with respect to the Lebesgue measure with a regular density, see Theorem~\ref{thm:concentration-Lebesgue}, Section~\ref{LlWZ3G}; (2) the deterministic setting where $\mu =  (\mathcal{H}^{N}\LL D) \times (\mathcal{H}^{N}\LL E)$, see Section~\ref{sec:det}; or (3) measures $\mu$ that are a product of an absolutely continuous measure with a deterministic measure of the form $\mu = (L_D \, \mathcal{L}^{2N}) \times (\mathcal{H}^{N}\LL E)$ or $\mu =(\mathcal{H}^{N}\LL D)\times (L_E \, \mathcal{L}^{2N})$, see Theorem~\ref{thm:det-Leb}.

{We introduce a concept of thermalization-concentration. Starting from a general measure on $Z\times Z$, we construct a sequence $\mu_\beta$ which  as $\beta\to\infty$ suppresses the contributions away from the diagonal. We first show that, under certain assumptions, $\mu_\beta$ converges to a limiting measure $\mu_\infty$ supported on the diagonal subset of $Z\times Z$, which one can of course identify with $Z$. We then (in Section \ref{secapproximation})  show  that this operation is, in a certain range, robust with respect to approximation. Specifically, for sequences of measures $\mu_h$ that approximates $\mu$, we shall identify sequences $(\mu_h)_{\beta_h}$, with $\beta_h\to\infty$, which converge to the same $\mu_\infty$. We stress that the measure $\mu$ characterizes the (practically, unknown) ``true'' material behavior, whereas $\mu_h$ are approximations that can (in principle) be obtained from sets of measurements.}

\subsection{Preliminaries}\label{sec:prelims}
For every $\mu \in \mathcal{M}(Z\times Z)$ and $\beta > 0$, define the corresponding thermalized measure as
\begin{equation}\label{v4coYU}
    \mu_\beta
    :=w_\beta \mu ,
    \quad
    w_\beta(y,z) := B_\beta^{-1} {\rm e}^{-\beta \|y-z\|^2}\,,\quad
    B_\beta
    :=
    \int_{Z}
        {\rm e}^{-\beta \|\xi\|^2}
    \, d\xi.
\end{equation}
We note that the scaling relation
\begin{equation}\label{eqBbetascal}
    B_\beta = B_{\beta_0} \left( \frac{\beta}{\beta_0} \right)^{-N} ,
    \qquad
    B_{\beta_0}
    =
    \int_{Z}
        {\rm e}^{-\beta_0 \|\xi\|^2}
    \, d\xi ,
\end{equation}
follows directly from one-homogeneity of the norm. Further, we consider weak convergence of Radon measures in $\mathcal M(X)$, denoted ${\mu_j} \rightharpoonup \mu$ or $\mu=\mathop{{w}{-}\lim}_{j \to \infty}\mu_j$, to mean
\begin{equation}\label{eqdefweakRadon}
    \int_X f\,d\mu_j\to\int_X f\,d\mu 
    \qquad \text{ as } j\to \infty 
    \qquad \text{ for all } f\in C_c(X)
\end{equation}
with $C_c(X)$ the set of continuous functions with compact support from the metric space $X$ to $\mathbb R$.  For $\mu\in\mathcal M(X)$ and $f\in C_c(X)$ we also write briefly $\mu(f):=\int_X fd\mu$. 

We denote by $\mathcal M_b(X):=\{\mu\in\mathcal M(X): \mu(X)<\infty\}$ the set of bounded Radon measures, and denote by the same symbol weak convergence of bounded (or probability) measures,
\begin{equation}\label{eqdefweakProb}
    \int_X f\,d\mu_j\to\int_X f\,d\mu 
    \qquad \text{ as } j\to \infty 
    \qquad \text{ for all } f\in C_b(X)\,,
\end{equation}
with $C_b(X)$ the set of continuous bounded functions from $X$ to $\mathbb R$. We clarify from the context which convergence is intended.
\begin{definition}[Diagonal concentration, transversality]\label{def:diag}
Let $\mu \in \mathcal{M}(Z\times Z)$ be such that:
\begin{itemize}
    \item[i)] (Boundedness)
    There is $\beta_*\ge 0$ such that
    $\mu_\beta \in \mathcal{M}_b(Z\times Z)$ for all $\beta > \beta_*$.
    \item[ii)] (Transversality) The weak limit
\begin{equation}\label{em4LG7}
    \mu_\infty := \mathop{{w}{-}\lim}_{\beta \to +\infty} \, \mu_\beta 
\end{equation}
exists in $\mathcal{M}_b(Z\times Z)$ (in the sense of \eqref{eqdefweakProb}).
\end{itemize}
Then, we call $\mu_\infty$ the diagonal concentration of $\mu$. We say that $\mu$ is transversal if it has a diagonal concentration.
\end{definition}

\begin{remark}
 \label{def:diag2}
If $\mu=\mu_D\times \mu_E$, then $\mu_\infty$ can be interpreted as the intersection $\mu_D\cap \mu_E$. If additionally $\mu_D=f\mathcal L^{2N}$ and  $\mu_E=g\mathcal L^{2N}$, with $f,g\in C^0(Z)$, then $\mu_\infty=\mu_D\cap \mu_E=fg\mathcal L^{2N}$, in agreement with the definition given in 
 \cite[Sect.~4.3.20]{Federer:1969} or
\cite[Sect.~3]{mattila1984hausdorff}.
\end{remark}

We verify that the diagonal concentration $\mu_\infty$, if it exits, is indeed supported on the diagonal.

\begin{lemma}\label{lemmadiagconc}
Let $\mu\in \mathcal{M}(Z\times Z)$ have a diagonal concentration $\mu_\infty$. Then, ${\rm supp}(\mu_\infty) \subseteq {\rm diag}(Z\times Z)$.
\end{lemma}

\begin{proof}
Fix any compact set $K\subset (Z\times Z)\setminus {\rm diag}(Z \times Z)$ with $\mu_\infty(\partial K)=0$ and let $\delta:={\rm dist}(K, {\rm diag }(Z \times Z))
=\min\{\vertiii{(y,z)-(w,w)}: (y,z)\in K, w\in Z\}$. Then $\|y-z\|\ge \delta$ for all $(y,z)\in K$, and we compute,
\begin{equation}
    \mu_\beta(K)
    =
    \int_K B_1^{-1}\beta^{{N}} {\rm e}^{-\beta\|y-z\|^2} d\mu(y,z)
    \le
    B_1^{-1}\beta^{{N}} {\rm e}^{-\beta \delta^2} \mu(K).
\end{equation}
Since $\mu_\infty(\partial K)=0$, we have
\begin{equation}
    \mu_\infty(K)
    =
    \lim_{\beta\to+\infty} \mu_\beta(K)
    \le
    \lim_{\beta\to+\infty}
    B_1^{-1}\beta^{{N}} {\rm e}^{-\beta \delta^2} \mu(K)
    =
    0 .
\end{equation}
The set $(Z\times Z)\setminus  {\rm diag }(Z \times Z)$ can be covered by countably many such sets $K$, and since $\mu_\infty$ is a Radon measure, $\mu_\infty(Z\times Z\setminus  {\rm diag }(Z \times Z))=0$.
\end{proof}

We also note that for purposes of characterizing $\mu_\infty$ it suffices to consider univariate test functions.

\begin{lemma}
Assume that $\mu$ has a diagonal concentration $\mu_\infty$, and assume that $\nu\in\mathcal M_b(Z\times Z)$ is such that $\nu(Z\times Z\setminus  {\rm diag }(Z \times Z))=0$ and that
\begin{equation}\label{eqtestuniva}
    \int_{Z\times Z} f(y) d\nu(y,z) =
    \int_{Z\times Z} f(y) d\mu_\infty(y,z) \quad \text{ for all } f\in C_b(Z).
\end{equation}
Then $\mu_\infty=\nu$.
\end{lemma}
\begin{proof}
Fix any $g\in C_b(Z\times Z)$. By $\nu(Z\times Z\setminus  {\rm diag }(Z \times Z))=0$, it follows that $g(y,z)=g(y,y)$ for $\nu$-almost every $(y,z)\in Z\times Z$. By Lemma \ref{lemmadiagconc}, the same holds for $\mu_\infty$. Using \eqref{eqtestuniva} with $f(y):=g(y,y)$, $f\in C_b(Z)$ leads to
\begin{equation}
\begin{split}
    \int_{Z\times Z} g(y,z) d\nu(y,z) =&
    \int_{Z\times Z} g(y,y) d\nu(y,z) \\
    =&
    \int_{Z\times Z} g(y,y) d\mu_\infty(y,z)
    =
    \int_{Z\times Z} g(y,z) d\mu_\infty(y,z) ,
\end{split}
\end{equation}
which concludes the proof.
\end{proof}

\subsection{Stochastic Loading}\label{LlWZ3G}
The following theorem provides an example of trans\-versality in the simple case in which $\mu$ is absolutely continuous with respect to the Lebesgue measure. For $\lambda>0$ consider the linear change of variables $S_\lambda:Z\times Z\to Z\times Z$ given by
\begin{equation}\label{hV5vEE}
    (y,z)
    =
    S_\lambda(\xi,\eta)
    :=
    \Big(
        \frac{\xi + \frac{\eta}{\lambda}}{\sqrt{2}} ,
        \frac{\xi - \frac{\eta}{\lambda}}{\sqrt{2}}
    \Big).
\end{equation}
This mapping can be inverted to give
\begin{equation}
    (\xi,\eta)
    =
    S_\lambda^{-1}(y,z)
    =
    \Big(
        \frac{y + z}{\sqrt{2}} ,
        \lambda \,
        \frac{y - z}{\sqrt{2}}
    \Big)\,.
\end{equation}
For $\beta_0>0$, we further introduce the notation
\begin{equation}\label{eqdefbarw}
     \bar w_{\beta_0}(\eta) := B_{\beta_0}^{-1} {\rm e}^{-2\beta_0 \|\eta\|^2}\,.
\end{equation}

\begin{theorem}\label{thm:concentration-Lebesgue}
Let $Z = \mathbb{R}^{2N}$ and $\mu \in \mathcal{M}(Z\times Z)$. Assume:
\begin{itemize}
\item[i)] (Regularity) $\mu = L \, \mathcal{L}^{2N} \times \mathcal{L}^{2N}$, $L \in L^1_{\rm loc}(Z\times Z)$ non-negative.
\item[ii)] (Continuity) There exists a function $\hat L:Z\to\mathbb{R}$ such that the pointwise limit
\begin{equation}
    \hat L\left(\frac{\xi}{\sqrt{2}}\right) =
    \lim_{\lambda\to+\infty}
    L\Big(S_\lambda(\xi,\eta)\Big) ,
\end{equation}
holds for $\mathcal L^{4N}$-a.~e.~$(\xi,\eta)\in Z\times Z$.
\item[iii)] (Equi-integrability) There is a function $g \in L^1(Z\times Z)$, $\lambda_0>0$ and $\beta_0 > 0$ such that
\begin{equation}
    \bar w_{\beta_0}(\eta)L\left(S_\lambda(\xi,\eta)\right)
    \leq
    g(\xi,\eta) 
\end{equation}
for $\mathcal L^{4N}$-a.~e.~$(\xi,\eta)\in Z\times Z$, and for all $\lambda \ge \lambda_0$.
\end{itemize}
Then, the weak limit $\mu_\infty$ of $\mu_\beta$ as $\beta \to +\infty$ exists in $\mathcal{M}_b(Z\times Z)$ (in the sense of \eqref{eqdefweakProb}) and
\begin{equation}\label{PIo3S4}
    \mu_\infty(f)
    =
    \int_Z
        f(\xi,\xi) \, \hat L(\xi)
    \, d\xi ,
\end{equation}
for all $f \in C_b(Z\times Z)$.
\end{theorem}
In other words, the conditions of Theorem~\ref{thm:concentration-Lebesgue}
guarantee that the measure $\mu$ has a diagonal concentration according to
Definition~\ref{def:diag}. The theorem also implies that the diagonal limit
$\hat L(\xi)$ is integrable.
\begin{proof}
By (i), for every $f \in C_c(Z\times Z)$ and $\lambda>0$ we have
\begin{align*}
    \mu_\beta(f)
    &=
    \int_{Z\times Z}
        f(y,z)
        B_\beta^{-1}
        {\rm e}^{-\beta \|y-z\|^2}
        L(y,z)
    \, dy \, dz \\
     &=
    \int_{Z\times Z}
        f\Big(S_\lambda(\xi,\eta)\Big)
        B_\beta^{-1}
        {\rm e}^{-2 \beta \lambda^{-2}\|\eta\|^2}
        L\Big(S_\lambda(\xi,\eta)\Big)
    \, \lambda^{-2N} d\xi \, d\eta\,,
\end{align*}
using the change of variables \eqref{hV5vEE}.
Let $\lambda
    =
    \sqrt{\beta/\beta_0}$,
with $\beta_0 > 0$ as in (iii).
With \eqref{eqdefbarw} and \eqref{eqBbetascal} we obtain
\begin{equation}\label{fR4kzk}
\begin{split}
    &
    \mu_\beta(f)
    = 
    \int_{Z\times Z}
        f\Big(S_\lambda(\xi,\eta)\Big)
        \bar w_{\beta_0}(\eta)
        L\Big(S_\lambda(\xi,\eta)\Big)
    \, d\xi \, d\eta .
\end{split}
\end{equation}
By (iii), if $\beta\ge \beta_*:=\beta_0\lambda_0^2$ then $\mu_\beta\in\mathcal M_b(Z\times Z)$ and the same representation holds for any $f\in C_b(Z\times Z)$. By (ii-iii) and the dominated convergence theorem,
\begin{equation}
    \lim_{\beta \to +\infty} \mu_\beta(f)
    =
    \int_{Z\times Z}
        f\Big(\frac{\xi}{\sqrt{2}}, \frac{\xi}{\sqrt{2}}\Big)
        \bar w_{\beta_0}(\eta)
        \hat L\Big(\frac{\xi}{\sqrt{2}}\Big)
    \, d\xi \, d\eta .
\end{equation}
Finally, by (iii) and Fubini's theorem,
\begin{equation}
    \lim_{\beta \to +\infty} \mu_\beta(f)
    =
    \int_Z
        f\Big(\frac{\xi}{\sqrt{2}}, \frac{\xi}{\sqrt{2}}\Big)
        \hat L\Big(\frac{\xi}{\sqrt{2}}\Big)
        2^{-N}
    \, d\xi
    =:
    \mu_\infty(f) .
\end{equation}
for any $f\in C_b(Z\times Z)$. The change of variables $\bar\xi:=\xi/\sqrt2$ leads to  \eqref{PIo3S4}.
\end{proof}

The following example further illustrates the transversality property of Definition~\ref{def:diag}.

\begin{example}[Sliding Gaussians]{\rm
With $y=(y_1,y_2) \in \mathbb{R}^N \times \mathbb{R}^N = Z$ and $z=(z_1,z_2) \in \mathbb{R}^N \times \mathbb{R}^N = Z$, let
\begin{equation}
    L_D(y)
    :=
    {\rm e}^{-\frac{1}{2} |a_1 y_1 + a_2 y_2|_N^2} ,
    \quad
    L_E(z)
    :=
    {\rm e}^{-\frac{1}{2} |b_1 z_1 + b_2 z_2|_N^2} ,
\end{equation}
with $a_1$, $a_2$, $b_1$, $b_2 \in \mathbb{R}$, $a_1^2+a_2^2 > 0$ and $b_1^2+b_2^2 > 0$, and $|\cdot|_n$ the Euclidean norm in $\mathbb R^n$ for $n=2,N$. We note that $L_D$ is constant on the $N$-dimensional  affine subspaces $\{a_1 y_1 + a_2 y_2 = {\rm constant}\}$ and is an (unnormalized) Gaussian on the complementary $N$-dimensional subspaces $\{(y_1,y_2) = (a_1 v,a_2 v),\, v\in \mathbb{R}^N\}$. Likewise, $L_E$ is constant on the $N$-dimensional affine subspaces $\{b_1 z_1 + b_2 z_2 = {\rm constant}\}$ and is an (unnormalized) Gaussian on the complementary $N$-dimensional subspaces $\{(z_1,z_2) =( b_1 w,b_2 w),\ w \in \mathbb{R}^N\}$. Thus, $\mu = L_D(y)L_E(z) \, \mathcal{L}^{2N} \times \mathcal{L}^{2N}$ is a non-negative Radon measure with density in $L^1_\mathrm{loc}(Z\times Z)$, but it is not finite. Condition ii) in Theorem~\ref{thm:concentration-Lebesgue} follows immediately from continuity of $L_D$ and $L_E$, and for $\xi=(\xi_1,\xi_2)\in \mathbb{R}^N\times\mathbb{R}^N$ we obtain
\begin{equation}
    \hat L(\frac{\xi}{\sqrt2})
    =
    L_D(\frac{\xi}{\sqrt2}) L_E(\frac{\xi}{\sqrt2})
    =
    {\rm e}^{-\frac{1}{4} \langle\xi,  Q\xi\rangle} ,
\end{equation}
with $Q\in\mathbb{R}^{2N\times 2N}$ given by
\begin{equation*}
    Q
    :=
    \begin{pmatrix}
        (a_1^2+b_1^2)I_N & (a_1a_2+b_1b_2)I_N \\
        (a_1a_2+b_1b_2)I_N & (a_2^2+b_2^2)I_N
    \end{pmatrix}\,.
\end{equation*}
Then, indeed, $\langle\xi,Q\xi \rangle = |a_1 \xi_1 + a_2 \xi_2|_N^2 + |b_1 \xi_1 + b_2 \xi_2|_N^2$. Suppose, specifically, that
\begin{equation}
    b_1 = a_1 \cos\theta - a_2 \sin\theta ,
    \quad
    b_2 = a_1 \sin\theta + a_2 \cos\theta ,
\end{equation}
i.~e., $\mu_D$ is rotated from $\mu_E$ by an angle $\theta$. Then, the eigenvalues of $Q$ follow as
\begin{equation}
    \lambda_{\rm min}
    =
    |a|_2^2 (1 - | \cos\theta |) ,
    \quad
    \lambda_{\rm max}
    =
    |a|_2^2 (1 + | \cos\theta |) ,
\end{equation}
with multiplicity $N$. Evidently, $\lambda_{\rm max} > 0$ for all $\theta$. In addition, $\lambda_{\rm min} > 0$, and $\hat L$ is integrable, if $\theta \not\in \pi\mathbb Z$, i.~e., if $\mu_D$ and $\mu_E$ are {\sl transverse}. Else, $\lambda_{\rm min} = 0$, and $\hat L$ is not integrable, if $\theta \in\pi\mathbb Z$, i.~e., if $\mu_D$ and $\mu_E$ are aligned.

We finally check that this example fulfills the assumptions of Theorem~\ref{thm:concentration-Lebesgue}. We already checked conditions (i) and (ii). In order to prove condition (iii), we choose $\lambda_0=1$, $\beta_0=(|a|_2^2+|b|_2^2)/4$, and observe that
\begin{equation}
\begin{split}
    \bar w_{\beta_0}(\eta)
    L(S_\lambda(\xi,\eta))
    = &
    B_{\beta_0}^{-1} {\rm e}^{-2\beta_0|\eta|_{2N}^2}
    \\ & \times 
    {\rm e}^{-\frac{1}{4} 
    |a_1 \xi_1 + a_2 \xi_2 + \frac{a_1\eta_1+a_2\eta_2}{\lambda}|_N^2
     -
     \frac{1}{4} |b_1 \xi_1 + b_2 \xi_2 -\frac{b_1\eta_1+b_2\eta_2}{\lambda}|_N^2} .
\end{split}
\end{equation}
Using that for any $x,y\in\mathbb R^N$ one has $|x+y|_N^2\ge \frac12 |x|_N^2-|y|_N^2$, we estimate
\begin{equation}
    \frac{1}{4} 
    |a_1 \xi_1 + a_2 \xi_2 + \frac{a_1\eta_1+a_2\eta_2}{\lambda}|_N^2
    \ge 
    \frac18 
    |a_1\xi_1+a_2\xi_2|_N^2-\frac1{4\lambda^2} |a|_2^2|\eta|_{2N}^2
\end{equation}
and the same for the second term. Using  $\lambda^2\ge 1$, this implies
\begin{equation}
\begin{split}
    \bar w_{\beta_0}(\eta)
    L(S_\lambda(\xi,\eta))
    \le &
    B_{\beta_0}^{-1} 
    {\rm e}^{-\beta_0|\eta|_{2N}^2-\frac{1}{8} |a_1 \xi_1 + a_2 \xi_2|_N^2
    -
    \frac{1}{8} |b_1 \xi_1 + b_2 \xi_2|_N^2} 
    =: 
    g(\xi,\eta) .
\end{split}
\end{equation}
If the vectors $a$ and $b$ are linearly independent (as elements of $\mathbb R^2$) the right-hand side is integrable and (iii) holds.
}\hfill$\square$
\end{example}

Applying $\mu_\beta \rightharpoonup \mu_\infty$ in $\mathcal{M}_b(Z\times Z)$ to the test function $f\equiv 1$ we see that the total variation passes to the limit.
\begin{corollary}\label{Ak5zZ5}
Under the assumptions of Theorem~\ref{thm:concentration-Lebesgue}, we have
\begin{equation}
    \lim_{\beta\to+\infty} |\mu_\beta| = |\mu_\infty|.
\end{equation}
\end{corollary}

From Theorem~\ref{thm:concentration-Lebesgue} and Corollary~\ref{Ak5zZ5} we immediately obtain convergence of expectations, i.~e.,
\begin{equation}\label{hqP8wE}
    \mathbb{E}_\infty[f]
    =
    \frac{\mu_\infty(f)}{|\mu_\infty|}
    =
    \lim_{\beta\to+\infty}
    \frac{\mu_\beta(f)}{|\mu_\beta|}
    =
    \lim_{\beta\to+\infty}
    \mathbb{E}_\beta[f] ,
\end{equation}
for every function $f \in C_b(Z\times Z)$.

\subsection{Deterministic Loading}\label{nXr8mw}
The case of deterministic loading is also of interest, making the link to the deterministic DD formulation discussed in Section~\ref{sec:det}.

\begin{theorem}\label{thm:det-Leb}
Let $Z = \mathbb{R}^{2N}$, $E$ an affine subspace of $Z$ of dimension $N$, $E_0$ the translate of $E$ through the origin, $E_0^\perp$ the orthogonal complement of $E_0$ and $\mu \in \mathcal{M}(Z\times Z)$. Assume:
\begin{itemize}
    \item[i)] (Regularity) $\mu = (L_D \, \mathcal{L}^{2N}) \times (\mathcal{H}^{N}\LL E)$, $L_D \in L^1_{\rm loc}(Z)$ non-negative.
    \item[ii)] (Continuity) The pointwise limit
\begin{equation}
    L_D(\xi)
    =
    \lim_{\lambda\to+\infty}
    L_D\Big(\xi+\frac{\eta}{\lambda}\Big) ,
\end{equation}
holds for $\mathcal H^N$-a.~e.~$\xi\in E$ and $\mathcal H^N$-a.~e.~$\eta \in E_0^\perp$.
    \item[iii)] (Equi-integrability) There is a function $g \in L^1(E\times E_0^\perp\times E_0; \mathcal{H}^{3N})$, $\lambda_0>0$ and $\beta_0 > 0$ such that
\begin{equation}
    {\rm e}^{-\beta_0 (\|\eta\|^2 + \|\zeta\|^2)}
    L_D\Big(\xi+\frac{\eta}{\lambda}\Big)
    \leq
    g(\xi,\eta,\zeta)
\end{equation}
for all $\lambda \ge \lambda_0$.
\end{itemize}
Then, the weak limit $\mu_\infty$ of $\mu_\beta$ as $\beta \to +\infty$ exists in $\mathcal{M}_b(Z\times Z)$ (in the sense of \eqref{eqdefweakProb}) and
\begin{equation}\label{AjC0tV}
    \mu_\infty(f)
    =
    \int_E
        f(\xi,\xi) \, L_D(\xi)
    \, d\mathcal H^N(\xi) ,
\end{equation}
for all $f \in C_b(Z\times Z)$.
\end{theorem}

In other words, we find that $\mu_\infty=L_D \mathcal{H}^{N}\LL (E,E)$, 
where $(E,E):=\{(z,z): z\in E\}$,
also see Fig.~\ref{lRkk0p}a. Note that the same argument can be applied to measures of the form $\mu =  (\mathcal{H}^{N}\LL D)\times (L_E \, \mathcal{L}^{2N})$.

\begin{proof}
By (i), we have, for every $f \in C_c(Z\times Z)$,
\begin{equation}\label{YuF3fv}
    \mu_\beta(f)
    =
    \int_{Z\times E}
        f(y,z)
        B_\beta^{-1}
        {\rm e}^{-\beta \|y-z\|^2}
        L_D(y)
    \, dy \, d\mathcal{H}^{N}(z) .
\end{equation}
With the aid of the orthogonal projection $P_E$ from $Z$ onto $E$, (\ref{YuF3fv}) can be rewritten as
\begin{equation}\label{t1yTvL}
    \mu_\beta(f)
    =
    \int_{Z\times E}
        f(y,z)
        B_\beta^{-1}
        {\rm e}^{-\beta \|y-P_E y\|^2}
        {\rm e}^{-\beta \|P_E y-z\|^2}
        L_D(y)
    \, dy \, d\mathcal{H}^{N}(z).
\end{equation}
Every $y\in Z$ can be uniquely represented as $y = \xi + \eta$, $\xi=P_Ey \in E$, $\eta=y-P_Ey \in E_0^\perp$. Once this is fixed, any $z\in E$ can be uniquely represented as $z = \xi - \zeta$ for a unique $\zeta \in E_0$. With this change of variables, (\ref{t1yTvL}) becomes
\begin{equation}\label{BO3pMY}
\begin{split}
    \mu_\beta(f)
    =
    \int_{E \times E_0^\perp \times E_0}
        f(\xi + \eta, \xi - \zeta)&
        B_\beta^{-1}
        {\rm e}^{-\beta \|\eta\|^2}
        {\rm e}^{-\beta \|\zeta\|^2}\\
        &\times
        L_D(\xi + \eta)
    \, d\mathcal H^N(\xi) \, d\mathcal H^N(\eta) \, d\mathcal H^N(\zeta).
    \end{split}
\end{equation}
Let $\beta_0$ be as in (iii), and $\lambda=\sqrt{\beta/\beta_0}$. Scaling $\eta$ and $\zeta$ by $\lambda$ we have
\begin{equation}\label{rrHL3A}
\begin{split}
    \mu_\beta(f)
    =
    \int_{E \times E_0^\perp \times E_0}
        f(\xi + \frac{\eta}{\lambda}, \xi - \frac{\zeta}{\lambda})&
        B_{\beta_0}^{-1}
        {\rm e}^{-\beta_0 \|\eta\|^2}
        {\rm e}^{-\beta_0 \|\zeta\|^2}
        \\ & \times        
        L_D(\xi + \frac{\eta}{\lambda})
    \, d\mathcal H^N(\xi) 
    \, d\mathcal H^N(\eta) 
    \, d\mathcal H^N(\zeta).
\end{split}
\end{equation}
By (iii), if $\lambda\ge \lambda_0$ (which is the same as $\beta\ge \beta_* := \beta_0 \lambda_0^2$) then $\mu_\beta\in\mathcal M_b(Z\times Z)$ and the same representation holds for every $f\in C_b(Z\times Z)$.

In view of (ii-iii), we can pass to the limit by Lebesgue dominated convergence, with the result
\begin{equation}
\begin{split}
    \lim_{\beta\to+\infty} \mu_\beta(f)
    =
    \int_{E \times E_0^\perp \times E_0}
        f(\xi, \xi)&
        B_{\beta_0}^{-1}
        {\rm e}^{-\beta_0 \|\eta\|^2}
        {\rm e}^{-\beta_0 \|\zeta\|^2}
        \\ & \times
        L_D(\xi)
    \, d\mathcal H^N(\xi) \, d\mathcal H^N(\eta) \, d\mathcal H^N(\zeta).
    \end{split}
\end{equation}
By (iii) and Fubini's theorem, an evaluation of the integrals with respect to $\eta$ and $\zeta$ finally gives
\begin{equation}
    \lim_{\beta\to+\infty} \mu_\beta(f)
    =
    \int_E
        f(\xi, \xi)
        L_D(\xi)
    \, d\mathcal H^N(\xi)
    =:
    \mu_\infty(f) .
\end{equation}
\end{proof}

Again, applying Theorem~\ref{thm:det-Leb} to the test function $f\equiv 1$, the total variation passes to the limit.

\begin{corollary}\label{axEQ2m}
Under the assumptions of Theorem~\ref{thm:det-Leb}, we have
\begin{equation}
    \lim_{\beta\to+\infty} |\mu_\beta| = |\mu_\infty|.
\end{equation}
\end{corollary}

As before, from Theorem~\ref{thm:det-Leb} and Corollary~\ref{axEQ2m} we obtain convergence of expectations in the sense of (\ref{hqP8wE}).

\subsection{Connection with the Kullback-Leibler divergence}\label{tw7QP8}
The term `thermalization' referring to $\mu_\beta$ may be motivated and
justified by the following connection to the Kullback-Leibler divergence, which is
known as the relative entropy in the analysis of partial differential equations.
Consider
fixed choices of $\mu \in
\mathcal{M}(Z\times Z)$ and $\beta>0$ such that $\mu_\beta\in
\mathcal{M}_b(Z\times Z)$. Then define the functional
$G_\beta:\mathcal{M}_b(Z\times Z) \to (-\infty,+\infty]$ as the Kullback-Leibler
divergence of $\nu$ with respect to $\mu_\beta$,
\begin{align}\label{SWwR3G}
    G_\beta(\nu)
    & :=
    \int_{Z\times Z} 
        \log\Big(
             \dfrac{d\nu}{d\mu_\beta}
        \Big) \, 
        d\nu
    -
    |\nu|+|\mu_\beta|
\end{align}
if $\nu \ll \mu_\beta$ and $\log \Big( \dfrac{d\nu}{d\mu_\beta}\Big)\in L^1(Z\times Z, \nu)$, and $G_\beta(\nu)=+\infty$ otherwise. 
If also
$\log \Big( {\dfrac{d\nu}{d\mu}}\Big)$ is integrable,
the expression above can be rewritten as
\begin{align}
    \beta
    \int_{Z\times Z}
        \|y-z\|^2\, d\nu(y,z)
            +\int_{Z\times Z}
            \log
            \Big(
              \dfrac{d\nu}{d\mu}
            \Big)
    \, d\nu
     + 
     (\log B_\beta-1) |\nu|+|\mu_\beta|.
\end{align}
This combines the Kullback-Leibler divergence of $\nu$ with respect to $\mu$ with an energy term $\beta \|y-z\|^2$. We recall that the Donsker-Varadhan variational formula \cite{Donsker:1975a} gives the representation
\begin{equation}\label{rNmNGn}
    G_\beta(\nu)
    =|\mu_\beta|+
    \sup_{f\in C_b(Z\times Z)}
    \Big(
        \int_{Z\times Z} f(y,z) \, d\nu(y,z)
        -|\nu|\log
        \int_{Z\times Z} {\rm e}^{f(y,z)} \, d\mu_\beta(y,z)
    \Big) ,
\end{equation}
with the supremum taken alternatively over all bounded measurable functions.

\begin{remark}
The Kullback-Leibler divergence of $\nu$ with respect to $\mu$ is usually defined for bounded measures only. This notion can be extended to unbounded $\mu\in\mathcal M(Z\times Z)$ as long as there exists some measurable function $W:Z\times Z\to [0,\infty)$ such that $\int_{Z\times Z} e^{-W}\,d\mu$ and $\int_{Z\times Z} W\,d\nu$ are finite, see \cite[Chapter 3]{leonard2014some}. In other words, $G_0$ is well-defined on bounded Radon measures with finite second moment as long as $\mu_\beta\in\mathcal{M}_b(Z\times Z)$ for some $\beta>0$.
\end{remark}

The main properties of the functionals $G_\beta$ are collected in the following result.

\begin{proposition}\label{bJTpsZ}
Let $Z = \mathbb{R}^N \times \mathbb{R}^N$ and $\mu \subset \mathcal{M}(Z\times Z)$ and $\beta > 0$. Then
\begin{itemize}
    \item[i)] $G_\beta$ is weakly lower-semicontinuous in $\mathcal{M}_b(Z\times Z)$.
    \item[ii)] $G_\beta$ is weakly coercive in $\mathcal{M}_b(Z\times Z)$.
    \item[iii)] The thermalized measure $\mu_\beta$ is the unique minimizer of $G_\beta$.
\end{itemize}
\end{proposition}

\begin{proof}

i) follows directly from the representation \eqref{rNmNGn}.

ii) The weak coercivity of $G_\beta$ in $\mathcal{M}(Z\times Z)$ follows directly from the weak coercivity of the Kullback-Leibler divergence.

iii) Consider the function $h:[0,\infty)\to\mathbb{R}$ defined by $h(x)=x\log(x)-x+1$ if $x>0$ and $h(0)=1$.  Then $G_\beta(\nu)=\int_{Z\times Z}h(d\nu/d\mu_\beta)\,d\mu_\beta$. Uniqueness of minimizers follows from strict convexity of $G_\beta$, which is a consequence of $h$ being strictly convex. As $h$ has a unique minimum at $x=1$, by Jensen's inequality, the unique minimizer satisfies $h(d\nu/d\mu_\beta)=1$ $\mu_\beta$-almost everywhere, which, since $\nu\ll\mu_\beta$, is the same as $\nu=\mu_\beta$.
\end{proof}

Observation iii) in Proposition~\ref{bJTpsZ} provides an interpretation of $\mu_\beta$ as an entropy regularized distribution, where the effect of $\beta$ appears via the entropy spreading out the distribution $\mu_\beta$. 

\section{Approximation}
\label{secapproximation}
Suppose now that the material behavior, geometry and loading are not known exactly, but only approximately through a sequence $(\mu_h)$ of measures converging to $\mu$ in $\mathcal{M}(Z\times Z)$ in some appropriate sense. It would then be natural to seek conditions under which the corresponding diagonal concentrations $(\mu_{h,\infty})$ converge to $\mu_\infty$ weakly in $\mathcal{M}_b(Z\times Z)$. However, a first challenge that impedes this program is that the diagonal concentrations $(\mu_{h,\infty})$ may not exist or be degenerate. For instance, suppose that the approximating measures are discrete and of the form
\begin{equation}\label{rPkN79}
    \mu_h
    =
    \sum_{i=1}^\infty
        c_{h,i} \delta_{(y_{h,i},z_{h,i})} ,
    \qquad
    c_{h,i} > 0,
    \qquad
    (y_{h,i},z_{h,i}) \in Z\times Z ,
\end{equation}
e.~g., resulting from empirical observation. In this case, the diagonal concentrations $(\mu_{h,\infty})$ are indeed likely to vanish generically. We overcome this difficulty by thermalizing the approximating measures $(\mu_h)$ in order to define an intermediate sequence $(\mu_{h,\beta_h})$, with $\beta_h\to+\infty$. By carefully choosing the quenching sequence $(\beta_h)$, we may expect to achieve the desired limit
\begin{equation}\label{gB6bTv}
    \mu_\infty = \mathop{{w}{-}\lim}_{h\to \infty} \mu_{h,\beta_h} ,
\end{equation}
in $\mathcal{M}_b(Z\times Z)$ in the sense of \eqref{eqdefweakProb}.

In this section, we endeavor to ascertain conditions under which the limit (\ref{gB6bTv}) is indeed attained. We begin by noting that (\ref{gB6bTv}) follows if
\begin{equation}\label{ysNW5O}
    \mu_\infty = \mathop{{w}{-}\lim}_{h\to \infty} \mu_{\beta_h} ,
\end{equation}
and, simultaneously,
\begin{equation}\label{x52Tto}
    \mathop{{w}{-}\lim}_{h\to \infty} \, (\mu_{h,\beta_h} - \mu_{\beta_h})
    =
    0 ,
\end{equation}
in $\mathcal{M}_b(Z\times Z)$. Eq.~(\ref{ysNW5O}) expresses the thermalization limit analyzed in Section~\ref{DOsT8G}. Conditions ensuring the convergence are provided by Theorems~\ref{thm:concentration-Lebesgue} and \ref{thm:det-Leb}. In the remainder of this section, we therefore turn attention to limit (\ref{x52Tto}).

\subsection{Random loading}\label{iw71LU}

We expect convergence to place restrictions on the rate at which the quenching schedule $(\beta_h)$ diverges to infinity. The following theorem puts forth conditions ensuring convergence in the case of measures $\mu$ absolutely continuous with respect to the Lebesgue measure.

\begin{theorem}\label{thm:randload}
Let $Z = \mathbb{R}^{2N}$ and $\mu \in \mathcal{M}(Z\times Z)$ and suppose that the assumptions of Theorem~\ref{thm:concentration-Lebesgue} hold. Let $(\mu_h)$ be a sequence of measures in $\mathcal{M}(Z\times Z)$. Assume, additionally, that:
\begin{itemize}
\item[iv)] For every $h \in \mathbb{N}$, there is a Borel transport map $T_h : Z\times Z \to Z\times Z$ such that
\begin{equation}\label{qMvw6F}
    \mu_h = T_h \# \mu .
\end{equation}
\item[v)] There is a sequence $(\beta_h)$ of positive numbers diverging to $+\infty$ such that, setting $\lambda_h:=\sqrt{\beta_h/\beta_0}$
with $\beta_0$ as in (iii), for every $(\xi,\eta) \in Z\times Z$ one has
\begin{equation}\label{gc0XtJ}
    \lim_{h\to\infty} T_h\circ S_{\lambda_h}(\xi,\eta)
    =\left(\frac\xi{\sqrt2},\frac\xi{\sqrt2}\right)
\end{equation}
and
\begin{equation}\label{eqconveta}
    \lim_{h\to\infty} \eta'_h(\xi,\eta)=\eta
\end{equation}
where we write
\begin{equation}\label{RDo9S0}
    (\xi'_h(\xi,\eta),\eta'_h(\xi,\eta)) := S_{\lambda_h}^{-1} \circ T_h \circ S_{\lambda_h}(\xi,\eta) .
\end{equation}
\item[vi)] There is a function $g\in L^1(Z\times Z)$ and $h_0\in\mathbb N$ such that for every $h\ge h_0$  and every $(\xi,\eta)\in Z\times Z$
\begin{equation}\label{aOI3Wq}
        \left[\bar w_{\beta_0}(\eta) +\bar w_{\beta_0}(\eta'_h(\xi,\eta))
 \right] L(S_{\lambda_h}(\xi,\eta))
        \le
        \, g(\xi,\eta).
\end{equation}
\end{itemize}
Then,
\begin{equation}
    \mathop{{w}{-}\lim}_{h\to\infty}
    \, (\mu_{h,\beta_h} - \mu_{\beta_h})
    =
    0 
\end{equation}
in $\mathcal M_b(Z\times Z)$ (in the sense of \eqref{eqdefweakProb}).
\end{theorem}
We recall that \eqref{qMvw6F} means that $\mu_h(A)=\mu(T_h^{-1}(A))$ for any Borel set $A\subset Z\times Z$ or, equivalently, $\mu_h(f)=\mu(f\circ T_h)$ for all $f\in C_c(Z\times Z)$. Also note that Assumption vi) above implies Assumption iii) from Theorem~\ref{thm:concentration-Lebesgue} along the sequence $\lambda_h$.

We illustrate with the following example how conditions \eqref{gc0XtJ} and  \eqref{eqconveta} can be verified in practice, and refer to Figure~\ref{fig1} for a geometrical interpretation.

\begin{figure}
\begin{center}
    \includegraphics[width=6cm]{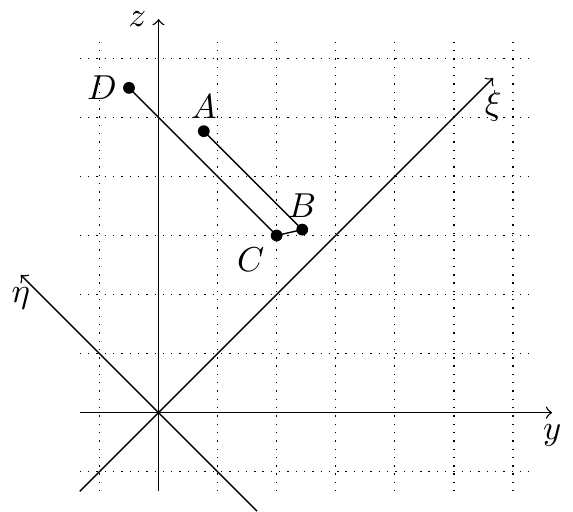}
\end{center}
    \caption{Sketch of the points appearing in \eqref{gc0XtJ}--\eqref{RDo9S0}. The point $A$ is $S_1(\xi,\eta)$. The mapping $S_\lambda\circ S_1^{-1}$ makes it closer to the diagonal, bringing it to $B:=S_\lambda(\xi,\eta)$. The map $T_h$ projects it onto the grid, leading to $C:=T_h(B)$. Finally, $D:=S_1(S_{\lambda^{-1}}(C))$ is again farther away from the diagonal (in this picture $\lambda=6$). Condition \eqref{gc0XtJ} states that $C$ is close to the diagonal for large $h$, condition \eqref{eqconveta} that $D$ is close to $A$, at least in the $\eta$-projection (which is the critical one for diagonal concentration}). \label{fig1}
\end{figure}

\begin{example}[Uniform grid]{\rm
Assume that $T_h$ is the projection onto $\delta_h\mathbb Z^{4N}$ (which is a componentwise operation), for some sequence $\delta_h\to0$. We denote by $p$ the projection of $\mathbb R^{2N}$ onto $\mathbb Z^{2N}$, defined componentwise by $p_i(y):=\lfloor y_i+\frac12\rfloor$, so that $T_h(y,z)=(\delta_h p(\delta_h^{-1}y), \delta_h p(\delta_h^{-1}z))$. Then
\begin{equation}
    T_h(S_{\lambda_h}(\xi,\eta))=(\delta_h
    p(\delta_h^{-1}\frac{\xi+\lambda_h^{-1}\eta}{\sqrt2}),
    \delta_hp( \delta_h^{-1}\frac{\xi-\lambda_h^{-1}\eta}{\sqrt2}))=:(y'_h,z'_h).
\end{equation}
From $\lambda_h\to\infty$ and $\delta_h\to0$ we obtain \eqref{gc0XtJ}. Further,
\begin{equation}
    \eta'_h(\xi,\eta)=\lambda_h\frac{y_h'-z_h'}{\sqrt2}
    =\frac{\lambda_h \delta_h}{\sqrt2}\left[
    p(\delta_h^{-1}\frac{\xi+\lambda_h^{-1}\eta}{\sqrt2})-p(\delta_h^{-1}
    \frac{\xi-\lambda_h^{-1}\eta}{\sqrt2})\right]
\end{equation}
so that  \eqref{eqconveta}  is equivalent to $\lambda_h\delta_h\to0$. This places a restriction on the rate at which the quenching schedule $\beta_h=\beta_0\lambda_h^2$  diverges depending on the rate of decay of the fineness of the discretization $\delta_h$.}
\end{example}

\begin{proof}[Proof of Theorem~\ref{thm:randload}]
Let $f\in C_c(Z\times Z)$. By iv), we have $\mu_{h,\beta_h}(f)=\mu_h(w_{\beta_h} f)=\mu( (w_{\beta_h} f)\circ T_h)$, so that
\begin{equation}
\begin{split}
    &
    \mu_{h,\beta_h}(f) - \mu_{\beta_h}(f)
    = \\ &
    \int_{Z\times Z}
        \Big(
            f(T_h(y,z)) \,  w_{\beta_h}(T_h(y,z))
            -
            f(y,z) \,  w_{\beta_h}(y,z)
        \Big)
        L(y,z)
    \, dy \, dz .
\end{split}
\end{equation}
Changing variables as in (\ref{hV5vEE}), we obtain
\begin{equation}
\begin{split}
    &
    \mu_{h,\beta_h}(f) - \mu_{\beta_h}(f)
    = \\ &
    \int_{Z\times Z}
        \Big(
            f(T_h\circ S_{\lambda_h}(\xi,\eta)) \,  w_{\beta_h}(T_h\circ S_{\lambda_h}(\xi,\eta))
            - \\ & \qquad\qquad\qquad\qquad
            f(S_{\lambda_h}(\xi,\eta)) \,  w_{\beta_h}(S_{\lambda_h}(\xi,\eta))
        \Big)
        L(S_{\lambda_h}(\xi,\eta))
    \, \lambda_h^{-2N} \, d\xi \, d\eta .
\end{split}
\end{equation}
Since $\lambda_h=\sqrt{\beta_h/\beta_0}$ with $\beta_0$ given by iii) in Theorem~\ref{thm:concentration-Lebesgue}, we have
\begin{equation}
    \lambda_h^{-2N}
    w_{\beta_h}(S_{\lambda_h}(\xi,\eta))
    =
    B_{\beta_0}^{-1} {\rm e}^{-2 \beta_0 \|\eta\|^2}
    =
    \bar w_{\beta_0}(\eta) ,
\end{equation}
and the same for {$(\xi,\eta)$ replaced with $(\xi_h',\eta_h')$}. We recall that \eqref{RDo9S0} implies $T_h\circ S_{\lambda_h}(\xi,\eta)=S_{\lambda_h}(\xi_h',\eta_h')$. Therefore
\begin{equation}\label{eqmuhbhfdiff}
\begin{split}
    &
    \mu_{h,\beta_h}(f) - \mu_{\beta_h}(f)
     \\ =&
    \int_{Z\times Z}
        \Big(
            f(T_h\circ S_{\lambda_h}(\xi,\eta)) \,
             \bar w_{\beta_0}(\eta_h')
            -
            f(S_{\lambda_h}(\xi,\eta)) \, \bar w_{\beta_0}(\eta)
        \Big)
        L(S_{\lambda_h}(\xi,\eta))
    \, d\xi \, d\eta .
\end{split}
\end{equation}
Analogously
\begin{equation}
\begin{split}
    &
    \mu_{h,\beta_h}(f) + \mu_{\beta_h}(f)
     \\ =&
    \int_{Z\times Z}
        \Big(
            f(T_h\circ S_{\lambda_h}(\xi,\eta)) \,
             \bar w_{\beta_0}(\eta')
            +
            f(S_{\lambda_h}(\xi,\eta)) \, \bar w_{\beta_0}(\eta)
        \Big)
        L(S_{\lambda_h}(\xi,\eta))
    \, d\xi \, d\eta.
\end{split}
\end{equation}
By assumption vi), for $h\ge h_0$ the integrand is bounded by $(\sup |f|) g(\xi,\eta)$, therefore $\mu_{h,\beta_h}$ and $\mu_{\beta_h}$ are in $\mathcal M_b(Z\times Z)$ and these formulas hold for all $f\in C_b(Z\times Z)$.

To conclude the proof we need to show that \eqref{eqmuhbhfdiff} converges to zero. By assumption ii) from Theorem~\ref{thm:concentration-Lebesgue}, for almost every $(\xi,\eta)$ the sequence $L(S_{\lambda_h}(\xi,\eta))$ converges. By continuity of $\bar w_{\beta_0}$ and \eqref{eqconveta} in assumption v), $\bar w_{\beta_0}(\eta')\to \bar w_{\beta_0}(\eta)$. By continuity of $f$ and \eqref{gc0XtJ} in assumption v), both $f(T_h\circ S_{\lambda_h}(\xi,\eta))$ and $f(S_{\lambda_h}(\xi,\eta))$ converge to $f(\xi/\sqrt2,\xi/\sqrt2)$. Therefore, the integrand in the right-hand side converges pointwise to zero. Since $f$ and $\bar w_{\beta_0}$ are bounded, using vi) and dominated convergence, we obtain
\begin{equation}
\begin{split}
    \lim_{h\to0}
    \left(    \mu_{h,\beta_h}(f) - \mu_{\beta_h}(f)\right)=0
\end{split}
\end{equation}
which concludes the proof.
\end{proof}

We expect condition vi) of Theorem~\ref{thm:randload} to place restrictions on the rate at which the quenching schedule $(\beta_h)$ diverges depending on the rate of decay of the likelihood function away from ${\rm diag}(Z \times Z)$. The following example illustrates the convergence conditions (v-vi) of Theorem~\ref{thm:randload}.

\begin{example}[Shifting error]\label{ZYe7s4}{\rm
Suppose that $\mu = L \, \mathcal{L}^{2N} \times \mathcal{L}^{2N}$ satisfies the conditions of Theorem~\ref{thm:concentration-Lebesgue} and $\mu_h = L_h \, \mathcal{L}^{2N} \times \mathcal{L}^{2N}$ contains errors with respect to $L$ by a shift $(u_h,v_h) \in Z \times Z$, i.~e.,
\begin{equation}
    L_h (y,z)
    =
    L(y-u_h,z-v_h) .
\end{equation}
In this case,
\begin{equation}
    T_h(y,z) = (y,z) + (u_h,v_h) ,
\end{equation}
which translates $L$ by $(u_h,v_h)$. Condition  \eqref{gc0XtJ} in v) of Theorem~\ref{thm:randload} requires that
\begin{equation}
    \lim_{h\to\infty}
    T_h\circ S_{\lambda_h}(\xi,\eta)
    =
    \lim_{h\to\infty}
    \Big(
        \frac{\xi + \frac{\eta}{\lambda_h}}{\sqrt{2}} + u_h ,
        \frac{\xi - \frac{\eta}{\lambda_h}}{\sqrt{2}} + v_h
    \Big)
    =
    (\frac\xi{\sqrt 2},\frac\xi{\sqrt2}) ,
\end{equation}
and, therefore, that
\begin{equation}\label{equhvh0}
    \lim_{h\to\infty} (u_h, v_h) = (0,0) .
\end{equation}
In addition, we compute
\begin{equation}\label{varprime}
    (\xi_h',\eta_h')=    S_{\lambda_h}^{-1} \circ T_h \circ S_{\lambda_h}(\xi,\eta)
    =
    \Big(
        \xi  + \left(\frac{u_h+v_h}{\sqrt{2}}\right) ,
        \eta + \lambda_h \left(\frac{u_h-v_h}{\sqrt{2}}\right)
    \Big).
\end{equation}
Therefore condition \eqref{eqconveta} requires
\begin{equation}
    \lim_{h\to\infty}
    \eta_h'
    =
    \eta + \lim_{h\to\infty}
    \lambda_h
        \frac{u_h-v_h}{\sqrt{2}}=\eta
\end{equation}
or, equivalently, that
\begin{equation}\label{equhvh1}
    \lim_{h\to\infty} \lambda_h(u_h-v_h) = 0 .
\end{equation}
It remains to check the uniform integrability condition vi). Assume that the function $\hat g:Z\to[0,\infty]$,
\begin{equation}\label{eqdefgxi}
    \hat g(\xi)
    :=
    \sup_{\eta\in Z} L\left(\frac{\xi+\eta}{\sqrt 2},\frac{\xi-\eta}{\sqrt 2}\right)
\end{equation}
is integrable (also see Remark~\ref{rmk:int} below).
We obtain, writing briefly $\eta'$ for $\eta'_h(\xi,\eta)$,
\begin{equation}\label{eqwbeta0etaLs1}
    \left[\overline w_{\beta_0}(\eta)+
    \overline w_{\beta_0}(\eta') \right]
    L(S_\lambda(\xi,\eta))\le
    \left[\overline w_{\beta_0}(\eta)+
    \overline w_{\beta_0}(\eta') \right]
    \hat g(\xi).
\end{equation}
From \eqref{varprime}, we have
$$
\|\eta\|^2 = \|\eta'-\lambda_h\left(\frac{u_h-v_h}{\sqrt{2}}\right)\|^2
\le 2\|\eta'\|^2 + \|\lambda_h\left(u_h-v_h\right)\|^2\,,
$$
where by \eqref{equhvh1} the last term decreases to zero uniformly with $h$, and so can be bounded by 2.
and therefore $\overline w_{\beta_0}(\eta')\le B_{\beta_0}^{-1} {\rm e}^{-\beta_0 \|\eta\|^2+2\beta_0}$, which is integrable over $Z$. Then the function
\begin{equation}
    g(\xi,\eta)
    :=
    2B_{\beta_0}^{-1} {\rm e}^{-\beta_0 \|\eta\|^2+2\beta_0} \hat g(\xi)
\end{equation}
is integrable  over $Z\times Z$ and shows that condition (vi) of Theorem~\ref{thm:randload} is satisfied. Evidently, \eqref{equhvh1} places a restriction on the quenching schedule $(\beta_h)$, which should diverge to $+\infty$ slower than $\| u_h - v_h \|^{-2}$.
\hfill$\square$}
\end{example}

\begin{remark}\label{rmk:int}
We remark that integrability of $\hat g$ as defined in \eqref{eqdefgxi} is related to, but different from, integrability of $L$ (which corresponds to the fact that $\mu$ is a bounded measure). For example, $L_a(y,z)={\rm e}^{-\|y+z\|^2}$ is not integrable, but $\hat g_a(\xi)={\rm e}^{-2\|\xi\|^2}$ is integrable. More generally, if $L(S_1(\xi,\eta))=L_\xi(\xi)L_\eta(\eta)$, with $L_\xi$ integrable and nonzero, then $L$ is integrable if and only if $L_\eta$ is. On the other hand,
boundedness of $L_\eta$ suffices to ensure integrability of $\hat g$.
\end{remark}

\subsection{Approximation by discrete empirical measures}\label{yu04L1}

Suppose that the approximating measure $\mu_h$ is of the form (\ref{rPkN79}). In this case, for every $f \in C_b(Z\times Z)$, we have
\begin{equation}
    \mu_{h,\beta_h}(f)
    =
    \sum_{i=1}^\infty
        f(y_{h,i},z_{h,i})
        c_{h,i}
        B_{\beta_h}^{-1}
        {\rm e}^{- \beta_h \|y_{h,i} - z_{h,i}\|^2} .
\end{equation}
The corresponding total variation is
\begin{equation}
    |\mu_{h,\beta_h}|
    =
    \sum_{i=1}^\infty
        c_{h,i}
        B_{\beta_h}^{-1}
        {\rm e}^{- \beta_h \|y_{h,i} - z_{h,i}\|^2} ,
\end{equation}
and the approximate expectation follows as
\begin{equation}\label{w635Xy}
    \mathbb{E}_{h,\beta_h}[f]
    =
    \frac
    {
        \sum_{i=1}^\infty
            f(y_{h,i},z_{h,i})
            c_{h,i}
            B_{\beta_h}^{-1}
            {\rm e}^{- \beta_h \|y_{h,i} - z_{h,i}\|^2}
    }
    {
        \sum_{i=1}^\infty
            c_{h,i}
            B_{\beta_h}^{-1}
            {\rm e}^{- \beta_h \|y_{h,i} - z_{h,i}\|^2}
    } .
\end{equation}
It bears emphasis that these approximate expectations are explicit in the data and involve no intermediate modeling step.

Theorem~\ref{thm:randload} supplies sufficient conditions for the approximate expectations (\ref{w635Xy}) to converge in the sense
\begin{equation}
    \lim_{h\to\infty}
    \mathbb{E}_{h,\beta_h}[f]
    =
    \mathbb{E}_\infty[f] \qquad \forall f \in C_b(Z\times Z)\,.
\end{equation}
In order to verify the assumptions of Theorem~\ref{thm:randload}, we begin by noting that the discrete empirical measure (\ref{rPkN79}) can be expressed in the form (\ref{qMvw6F}) by introducing a Borel transport map taking discrete values $T_h:Z\times Z\to D_h$, with $D_h = \{(y_{h,i},z_{h,i})\}_{i=1}^\infty$, and setting
\begin{equation}
    A_{h,i} := T_h^{-1}(y_{h,i},z_{h,i}) ,
    \qquad
    c_{h,i} := \mu_h(A_{h,i})=\int_{A_{h,i}} L(y,z) dydz .
\end{equation}
We assume that the sets $(A_{h,i})_{i\in\mathbb N}$ are Borel and constitute a partition of $Z\times Z$. We also assume that the approximation becomes asymptotically finer, in a sense that will be made precise below in \eqref{eqlambdahdeltah}, and that the limiting measure $\mu=L\mathcal L^{4N}$ obeys the integrability property \eqref{eqhatgint1234}.
We in particular assume that $c_{h,i}<\infty$, which is guaranteed if each $A_{h,i}$ is bounded. In addition, we define the displacement $(u_h(y,z), v_h(y,z))$ that takes $(y,z)$ to $D_h$ by
\begin{equation}
    ( u_h(y,z), v_h(y,z) ):=    T_h(y,z)-
    (y,z).
\end{equation}
Proceeding as in Example \ref{ZYe7s4}, we obtain
\begin{equation}
\begin{split}
    T_h\circ S_{\lambda_h}(\xi,\eta)
    & =
    \Big(
        \frac{\xi + \frac{\eta}{\lambda_h}}{\sqrt{2}} ,
        \frac{\xi - \frac{\eta}{\lambda_h}}{\sqrt{2}}
    \Big)
    \\ & +
    \left(
        u_h
        \Big(
            \frac{\xi + \frac{\eta}{\lambda_h}}{\sqrt{2}} ,
            \frac{\xi - \frac{\eta}{\lambda_h}}{\sqrt{2}}
        \Big) ,
        v_h
        \Big(
            \frac{\xi + \frac{\eta}{\lambda_h}}{\sqrt{2}} ,
            \frac{\xi - \frac{\eta}{\lambda_h}}{\sqrt{2}}
        \Big)
    \right) ,
\end{split}
\end{equation}
and assumption \eqref{gc0XtJ} of Theorem~\ref{thm:randload} is satisfied if
\begin{equation}\label{X4sfD3}
    \lim_{h\to\infty}
    u_h
    \Big(
        \frac{\xi + \frac{\eta}{\lambda_h}}{\sqrt{2}} ,
        \frac{\xi - \frac{\eta}{\lambda_h}}{\sqrt{2}}
    \Big)
    =
    0 ,
    \quad
    \lim_{h\to\infty}
    v_h
    \Big(
        \frac{\xi + \frac{\eta}{\lambda_h}}{\sqrt{2}} ,
        \frac{\xi - \frac{\eta}{\lambda_h}}{\sqrt{2}}
    \Big)
    =
    0 ,
\end{equation}
for any $(\xi,\eta) \in Z \times Z$. In addition, proceeding as in Example \ref{ZYe7s4}, we compute
\begin{equation}\label{V2mbpe}
    \| \eta' - \eta \|
    =
    \frac{\lambda_h}{\sqrt{2}}
    \left\|
        u_h
        \Big(
            \frac{\xi + \frac{\eta}{\lambda_h}}{\sqrt{2}} ,
            \frac{\xi - \frac{\eta}{\lambda_h}}{\sqrt{2}}
        \Big)
        -
        v_h
        \Big(
            \frac{\xi + \frac{\eta}{\lambda_h}}{\sqrt{2}} ,
            \frac{\xi - \frac{\eta}{\lambda_h}}{\sqrt{2}}
        \Big)
    \right\| ,
\end{equation}
so that \eqref{eqconveta} places restrictions on the quenching schedule
$(\beta_h)$.

In order to make these conditions more explicit, assume that the cells $A_{h,i}$ contain the corresponding points $(y_{h,i},z_{h,i})$, and denote by $\delta_h(y,z)$ the diameter of the cell $A_{h,i}$ containing $(y,z)$. Then, we have
\begin{equation}
\vertiii{
        (u_h,v_h)
        \left(
            \frac{\xi + \frac{\eta}{\lambda_h}}{\sqrt{2}} ,
            \frac{\xi - \frac{\eta}{\lambda_h}}{\sqrt{2}}
        \right)}
    \leq
    \delta_h\left(S_{\lambda_h}(\xi,\eta)\right) ,
\end{equation}
and (\ref{X4sfD3}) follows if
\begin{equation}\label{eqconvdeltah13}
    \lim_{h\to\infty}
    \delta_h\left(S_{\lambda_h}(\xi,\eta)\right) = 0
\end{equation}
pointwise in $\xi$ and $\eta$. This condition requires, in particular, that $\delta_h(\xi,\xi) \to 0$, i.~e., that the point set become infinitely dense in the limit on ${\rm diag}(Z\times Z)$, and, for given $(\beta_h)$ it places restrictions on how sparse the point-data density can be away from ${\rm diag}(Z\times Z)$, i.~e., as $y$ and $z$ become decorrelated. In addition, from (\ref{V2mbpe}) we have
\begin{equation}\label{eqdiffetaetap}
    \| \eta' - \eta \|
    \leq
    \sqrt{2} \lambda_h
    \delta_h(S_{\lambda_h}(\xi,\eta)) ,
\end{equation}
so that the condition
\begin{equation}\label{eqlambdahdeltah}
    \lim_{h\to\infty} \lambda_h
    \delta_h\left(S_{\lambda_h}(\xi,\eta)\right)
    = 0
\end{equation}
ensures that both \eqref{gc0XtJ} and \eqref{eqconveta} are satisfied. 
It remains to check (vi). We proceed as above and define $\hat g:Z\to[0,\infty]$ by
\begin{equation}
    \hat g(\xi)
    :=
    \sup_{\eta\in Z} L\left(\frac{\xi+\eta}{\sqrt 2},\frac{\xi-\eta}{\sqrt 2}\right).
\end{equation}
We assume integrability,
\begin{equation}\label{eqhatgint1234}
 \hat g\in L^1(Z\times Z).
\end{equation}
In order to ensure that $g$ obeys the bound \eqref{aOI3Wq}, we assume that \eqref{eqlambdahdeltah} holds uniformly, in the sense that
\begin{equation}\label{eqlambdahdeltahunif}
     \lim_{h\to\infty}  \lambda_h   \sup_{y,z}
    \delta_h\left(y,z\right)
    = 0 .
\end{equation}
Then there is $h_0\in\mathbb N$ such that $\lambda_h\ge1$ and $$|\lambda_h\delta_h(S_{\lambda_h}(\xi,\eta))|\le \frac{1}{\sqrt2} \qquad \forall h\ge h_0, \forall (\xi,\eta)\,. $$
By \eqref{eqdiffetaetap} this implies $\|\eta'-\eta\|\le 1$. The rest of the argument is as in 
Example \ref{ZYe7s4}: we define
\begin{equation}
    g(\xi,\eta)
    :=
    2B_{\beta_0}^{-1} {\rm e}^{-\beta_0 \|\eta\|^2+2\beta_0} \hat g(\xi),
\end{equation}
and observe that 
\eqref{eqhatgint1234} implies
$g\in L^1(Z\times Z)$. The proof of \eqref{aOI3Wq} is the same as in \eqref{eqwbeta0etaLs1}. We stress that the assumption \eqref{eqlambdahdeltahunif} requires, in particular, $\sqrt{\beta_h}$ to diverge more slowly than $1/\delta_h$.

\subsection{Deterministic loading}\label{lH1P1g}

The case of deterministic loading is amenable to further simplification. Let $Z=\mathbb R^{2N}$. For a given affine subspace $E$ of $Z$ of dimension $N$, with $E_0$ the translate of $E$ through the origin and $E_0^\perp$ the orthogonal complement of $E_0$, we introduce the mapping $S_\lambda : E \times E_0^\perp \times E_0 \to Z \times E$
\begin{equation}\label{tRn1Y8}
    (y,z)
    =
    S_\lambda(\xi,\eta,\zeta)
    :=
    \Big(
        \xi + \frac{\eta}{\lambda} ,
        \xi - \frac{\zeta}{\lambda}
    \Big) ,
\end{equation}
for shorthand. This mapping can be inverted to give
\begin{equation}\label{tRn1Y8b}
    (\xi,\eta,\zeta)
    =
    S_\lambda^{-1}(y,z)
    =
    \Big(
        P_E y ,
        \lambda \,
        (y - P_E y) ,
        \lambda \,
        (P_E y - z)
    \Big) ,
\end{equation}
with $(y,z) \in Z \times E$ and $P_E$ the orthogonal projection of $Z$ onto $E$, see Figure~\ref{fig2}.

\begin{figure}
\begin{center}
    \includegraphics[width=6cm]{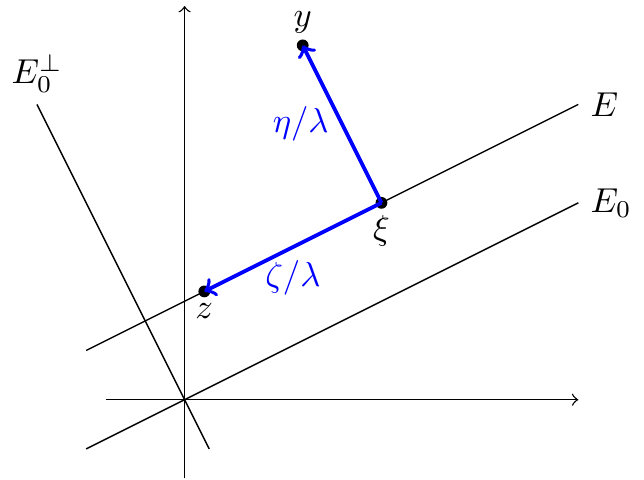}
\end{center}
    \caption{Sketch of the change of variables in \eqref{tRn1Y8}-\eqref{tRn1Y8b}.} \label{fig2}
\end{figure}

\begin{theorem}\label{tQ4UnZ}
Suppose that the assumptions of Theorem~\ref{thm:det-Leb} hold. Let $(\mu_{D,h})$ be a sequence of measures in $\mathcal{M}(Z)$,
 and $\mu_h:=\mu_{D,h}\times (\mathcal H^N\LL E)\in \mathcal M(Z\times Z)$. Assume, additionally, that:
\begin{itemize}
    \item[iv)] For every $h \in \mathbb{N}$, there is a Borel transport map $T_{D,h} : Z \to Z$ such that
\begin{equation}\label{T6gPqa}
    \mu_{D,h} = T_{D,h} \# \mu_D ,
\end{equation}
where $\mu_D = L_D \, \mathcal{L}^{2N}$. We write $ T_h(y,z) = (T_{D,h}(y),z)$.
    \item[v)] For every $(\xi,\eta,\zeta) \in E\times E_0^\perp \times E_0$,
\begin{equation}\label{gc0XtJb}
    \lim_{h\to\infty} S_{\lambda_h}^{-1}\circ T_h\circ S_{\lambda_h}(\xi,\eta,\zeta)
    =
    \left(\xi,\eta,\zeta\right) .
\end{equation}
\item[vi)] There is a sequence $(\beta_h)$ of positive numbers diverging to $+\infty$ and $g\in L^1(E\times E_0^\perp \times E_0;\mathcal H^{3N})$ such that for $h\ge h_0$
\begin{equation}\label{UpF06E}
\begin{split}
    \left[ 
        \hat w_{\beta_0}(\eta,\zeta)
        +
        \hat w_{\beta_0}(\eta_h',\zeta_h')
    \right] 
    L_D\left(\xi+\frac{\eta}{\lambda_h}\right) 
    \le
    g(\xi,\eta,\zeta) ,
\end{split}
\end{equation}
where we write
$\lambda_h:=\sqrt{\beta_h/\beta_0}$, with $\beta_0$ as in (iii), and
\begin{equation}\label{sD3vIL}
    (\xi'_h,\eta'_h,\zeta'_h) := S_{\lambda_h}^{-1} \circ T_h \circ S_{\lambda_h}(\xi,\eta,\zeta) 
\end{equation}
(depending implicitly on $\xi$, $\eta$, $\zeta$) 
 and, for $(\eta,\zeta)\in E_0^\perp\times E_0=Z$,
\begin{equation}
    \hat w_{\beta_0}(\eta,\zeta) :=     B_{\beta_0}^{-1} {\rm e}^{-\beta_0
    \|\eta\|^2-\beta_0 \|\zeta\|^2}.
\end{equation}
\end{itemize}
Then,
\begin{equation}
    \mathop{{w}{-}\lim}_{h\to\infty}
    (\mu_{h,\beta_h} - \mu_{\beta_h})
    =
    0
\end{equation}
in $\mathcal M_b(Z\times Z)$ (in the sense of \eqref{eqdefweakProb}).
\end{theorem}

\begin{proof} Let $f\in C_c(Z\times Z)$. By (iv), we have
\begin{equation}
\begin{split}
    &
    \mu_{h,\beta_h}(f) - \mu_{\beta_h}(f)
    = \\ &
    \int_{Z\times E}
        \Big(
            f(T_h(y,z)) \,  w_{\beta_h}(T_h(y,z))
            -
            f(y,z) \,  w_{\beta_h}(y,z)
        \Big)
        L_D(y)
    \, dy \, d\mathcal H^N(z) .
\end{split}
\end{equation}
Changing variables as in (\ref{tRn1Y8}), we obtain
\begin{equation}
\begin{split}
    &
    \mu_{h,\beta_h}(f) - \mu_{\beta_h}(f)
    = \\ &
    \int_{E \times E_0^\perp \times E_0}
        \Big(
            f(T_h\circ S_{\lambda_h}(\xi,\eta,\zeta)) \,  w_{\beta_h}(T_h\circ S_{\lambda_h}(\xi,\eta,\zeta))
            - \\ & \qquad\qquad
            f(S_{\lambda_h}(\xi,\eta,\zeta)) \,  w_{\beta_h}(S_{\lambda_h}(\xi,\eta,\zeta))
        \Big)
        L_D\left(\xi + \frac{\eta}{\lambda_h}\right)
        \, \lambda_h^{-2N} 
    \, d\mathcal H^N(\xi) 
    \, d\mathcal H^N(\eta) 
    \, d\mathcal H^N(\zeta) .
\end{split}
\end{equation}
Using that $\lambda_h=\sqrt{\beta_h/\beta_0}$ with $\beta_0$ as in iii), then
\begin{equation*}
    \lambda_h^{-2N}
    w_{\beta_h}(S_{\lambda_h}(\xi,\eta,\zeta))
    =
    B_{\beta_0}^{-1} 
    {\rm e}^{-\beta_0 \|\eta\|^2} 
    {\rm e}^{-\beta_0 \|\zeta\|^2}
    =:
    \hat w_{\beta_0}(\eta,\zeta) ,
\end{equation*}
and, therefore,
\begin{equation}\label{end-est}
\begin{split}
    &
    \mu_{h,\beta_h}(f) - \mu_{\beta_h}(f)
    = \\ &
    \int_{E \times E_0^\perp \times E_0}
        \Big(
            f(T_h\circ S_{\lambda_h}(\xi,\eta,\zeta)) \,
            \hat w_{\beta_0}(\eta_h',\zeta_h')
            - \\ & \qquad\qquad\qquad\qquad
            f(S_{\lambda_h}(\xi,\eta,\zeta)) \,  \hat w_{\beta_0}(\eta,\zeta)
        \Big)
        L_D\left(\xi + \frac{\eta}{\lambda_h}\right)
    \,  d\mathcal H^N(\xi) \, d\mathcal H^N(\eta) \,
d\mathcal H^N(\zeta) .
\end{split}
\end{equation}
As in the proof of Theorem \ref{thm:randload}, a similar computation and (vi) ensure that $\mu_{h,\beta_h}$ and $\mu_{\beta_h}$ are bounded measures for $h\ge h_0$, and so \eqref{end-est} holds for all $f\in C_b(Z\times Z)$. Further, by assumptions (ii) and (v), the integrand in the right-hand side converges pointwise to zero and the claim follows from (iii), (vi) and Lebesgue's dominated convergence theorem.
\end{proof}

\subsection{Approximation of the material likelihood by discrete empirical measures}\label{8yK9TU}

Suppose that the material likelihood measure $\mu_D \in \mathcal{M}(Z)$ is approximated by discrete empirical measures of the form
\begin{equation}\label{Pakqbt}
    \mu_{D,h}
    =
    \sum_{i=1}^{\infty}
        c_{h,i} \delta_{y_{h,i}} ,
    \quad
    c_{h,i} \geq 0,
    \quad
    y_{h,i} \in Z ,
\end{equation}
where $(y_{h,i})$ a point data sets, possibly finite, and $c_{h,i}$ is the likelihood of data point $y_{h,i}$. Consider the measure $\mu_h=\mu_{D,h}\times (\mathcal H^N\LL E)$. From (\ref{Pakqbt}), the approximate likelihood of outcomes of a univariate quantity of interest $f \in C_c(Z)$ evaluates to
\begin{equation}\label{skse4A}
    \int_{Z\times Z} f(z) \, d\mu_{h,\beta_h}(y,z)
    =
    \sum_{i=1}^{\infty}
        c_{h,i}
        \int_E
            f(z) B_{\beta_h}^{-1} {\rm e}^{-\beta_h \|y_{h,i}-z\|^2}
        \, d\mathcal H^N(z) .
\end{equation}
This expression may be simplified by recourse to the closest-point projection $P_E$ from $Z$ onto $E$. Denoting
\begin{equation}
    z_{h,i}
    :=
    P_E (y_{h,i})
    \in E ,
\end{equation}
for all points in the material data set and decomposing the vectors $(y_{h,i}-z)$ into normal and parallel components with respect to $E$, (\ref{skse4A}) reduces to
\begin{equation}\label{8agtjW}
\begin{split}
    &
    \int_{Z\times Z} f(z)  d\mu_{h,\beta_h}(y,z)
    = \\ &
    \sum_{i=1}^{\infty}
    c_{h,i} B_{\beta_h}^{-1}
    {\rm e}^{-\beta_h \|y_{h,i}-z_{h,i}\|^2}
    \Big(
        \int_E
            f(z)
            {\rm e}^{-\beta_h \|z-z_{h,i}\|^2}
        \, d\mathcal H^N(z)
    \Big) .
\end{split}
\end{equation}
Let $E_0$ be the translate of $E$ through the origin. Then,
\begin{equation}\label{HMJmfY}
    \int_E
        f(z)
        {\rm e}^{-\beta_h \|z-z_{h,i}\|^2}
    \,  d\mathcal H^N(z)
    =
    \int_{E_0}
        f(z_{h,i}+\xi)
        {\rm e}^{-\beta_h \|\xi\|^2}
    \,  d\mathcal H^N(\xi)
   =
    C_{\beta_h} f_{h,i} ,
\end{equation}
with
\begin{equation}
    f_{h,i}
    :=
    C_{\beta_h}^{-1}
    \int_{E_0}
        f(z_{h,i}+\xi)
        {\rm e}^{-\beta_h \|\xi\|^2}
    \, d\mathcal H^N(\xi) ,
    \quad
    C_{\beta_h}
    :=
    \int_{E_0}
        {\rm e}^{-\beta_h \|\xi\|^2}
    \, d\mathcal H^N(\xi) ,
\end{equation}
and (\ref{8agtjW}) reduces to
\begin{equation}\label{PEj3fn}
    \int_{Z\times Z} f(z)  \, d\mu_{h,\beta_h}(y,z)
    =
    B_{\beta_h}^{-1} C_{\beta_h}
    \sum_{i=1}^{\infty}
    c_{h,i} f_{h,i}
    {\rm e}^{-\beta_h \|y_{h,i}-z_{h,i}\|^2} ,
\end{equation}
which is explicit up to quadratures over $E_0$. In particular, with $f=1$,
(\ref{PEj3fn}) gives
\begin{equation}\label{lGa4W5}
    | \mu_{D,h,\beta_h} |
    =
    B_{\beta_h}^{-1} C_{\beta_h}
    \sum_{i=1}^{\infty}
    c_{h,i}
    {\rm e}^{-\beta_h \|y_{h,i}-z_{h,i}\|^2} .
\end{equation}
If this sum is nonzero and finite, from these identities, the approximate
expectation of outcomes for $f$ follows as
\begin{equation}\label{G6Q9ug}
    \mathbb{E}_h[f]
    =
    \frac
    {
        \sum_{i=1}^{\infty}
        c_{h,i} f_{h,i}
        {\rm e}^{-\beta_h \|y_{h,i}-z_{h,i}\|^2}
    }
    {
        \sum_{i=1}^{\infty}
        c_{h,i}
        {\rm e}^{-\beta_h \|y_{h,i}-z_{h,i}\|^2}
    } .
\end{equation}
Again, it bears emphasis that these approximate expectations are explicit in the data and involve no intermediate modeling step.

As in the case of random loading, Theorem~\ref{tQ4UnZ} sets forth sufficient conditions for the approximate expectations (\ref{G6Q9ug}) to converge to $\mathbb{E}_\infty[f]$. In order to make such convergence conditions more explicit, suppose that the transport map $T_{D,h}$ introduced in (\ref{T6gPqa}) takes values in the set $D_h = \{y_{h,i}\}_{i=1}^\infty $. Let
\begin{equation}
    A_{h,i} := T_h^{-1}(y_{h,i}) ,
    \qquad
    c_{h,i} := \mu_{D,h}(A_{h,i}) ,
\end{equation}
and assume that the sets $A_{h,i}$ are bounded Borel sets and that $(A_{h,i})_{i\in\mathbb N}$ forms a partition of $Z$, which  becomes finer with increasing $h$ in the sense of \eqref{eqlahdehb} below, and that the limiting measure is integrable in the sense of \eqref{eqhatgl1due} below.
We write
\begin{equation}
    T_{D,h}(y)
    =
    y
    +
    u_h(y) .
\end{equation}
Then, a simple calculation gives
\begin{equation}
    T_h\circ S_{\lambda_h}(\xi,\eta,\zeta)
    =
    \Big(
        \xi + \frac{\eta}{\lambda_h} ,
        \xi - \frac{\zeta}{\lambda_h}
    \Big)
    +
    \left(
        u_h
        \Big(
            \xi + \frac{\eta}{\lambda_h}
        \Big) ,
        0
    \right) ,
\end{equation}
and
\begin{equation}
    S_{\lambda_h}^{-1}\circ T_h\circ S_{\lambda_h}(\xi,\eta,\zeta)
    =
    (\xi,\eta,\zeta)
    +
    \Big(P_E u_h, \lambda_h(u_h-P_Eu_h), \lambda_h P_E u_h\Big),
\end{equation}
with $u_h$ evaluated at $\xi+\frac\eta{\lambda_h}$. Assumption (v) of Theorem~\ref{tQ4UnZ} is satisfied if
\begin{equation}\label{H88xX2}
    \lim_{h\to\infty}
 \lambda_h   u_h
    \Big(
        \xi + \frac{\eta}{\lambda_h}
    \Big)
    =
    0
\end{equation}
for all $(\xi,\eta) \in E \times E_0^\perp$. This results in restrictions on the quenching schedule $(\beta_h)$.

In order to make these conditions more explicit, assume that the cells $A_{h,i}$ contain the corresponding points $y_{h,i}$ and denote by $\delta_h(y)$ the diameter of the cell $A_{h,i}$ containing $y$. Then, we have
\begin{equation}
    \| u_h \Big( \xi + \frac{\eta}{\lambda_h} \Big) \|
    \leq
    \delta_h \Big( \xi + \frac{\eta}{\lambda_h} \Big)
\end{equation}
and so (\ref{H88xX2}) reduces to showing that
\begin{equation}\label{eqlahdeh}
    \lim_{h\to\infty} 
    \lambda_h \delta_h 
    \Big( \xi + \frac{\eta}{\lambda_h} \Big) 
    = 
    0 .
\end{equation}
This condition requires, in particular, that $\delta_h(\xi) \to 0$, i.~e., that the point set becomes infinitely dense in the limit on $E$, and, for given $(\beta_h)$ it places restrictions on how sparse the point-data density can be away from $E$.

It remains to verify assumption (vi). We proceed as in the previous examples, define $\hat g:E\to[0,\infty]$ by
\begin{equation}
    \hat g(\xi):=\sup_{\eta\in E_0^\perp} L_D(\xi+\eta)
\end{equation}
and assume 
\begin{equation}\label{eqhatgl1due}
\hat g\in L^1(E;\mathcal H^N).  
\end{equation}
We assume that \eqref{eqlahdeh} holds uniformly, in the sense that
\begin{equation}\label{eqlahdehb}
    \lim_{h\to\infty} \lambda_h \sup_{y\in Z}\delta_h ( y) = 0\,.
\end{equation}
Select $h_0$ such that $\lambda_h \sup_{y\in Z}\delta_h ( y)\le 1$ for all $h\ge h_0$, which implies 
$\|\eta'-\eta\|\le1$, 
and hence $-\|\eta'\|^2\le -\frac12\|\eta\|^2+1$. The same holds for $\zeta'$.
Therefore
\begin{equation}
    \hat w_{\beta_0} (\eta',\zeta')\le
    B_{\beta_0}^{-1} {\rm e}^{-\frac12 \beta_0 (\|\eta\|^2+\|\zeta\|^2)+2\beta_0}.
\end{equation}
We then define
\begin{equation}
    g(\xi,\eta,\zeta):=2B_{\beta_0}^{-1} {\rm e}^{-\frac12 \beta_0
    (\|\eta\|^2+\|\zeta\|^2)+2\beta_0} \hat g(\xi),
\end{equation}
and observe that integrability of $\hat g$ over $E$, which we assumed in \eqref{eqhatgl1due}, implies integrability of $g$ over $E\times E_0^\perp\times E_0$.

\section*{Acknowledgements}

This work was funded by the Deutsche Forschungsgemeinschaft (DFG, German Research Foundation) {\sl via} project 211504053 - SFB 1060; project 441211072 - SPP 2256; and project 390685813 -  GZ 2047/1 - HCM.

\bibliography{biblio}

\begin{thebibliography}{10}
\providecommand{\url}[1]{{#1}}
\providecommand{\urlprefix}{URL }
\expandafter\ifx\csname urlstyle\endcsname\relax
  \providecommand{\doi}[1]{DOI \discretionary{}{}{}#1}\else
  \providecommand{\doi}{DOI \discretionary{}{}{}\begingroup
  \urlstyle{rm}\Url}\fi

\bibitem{Truesdell:1960}
C.~Truesdell, R.A. Toupin, \emph{The Classical Field Theories} (Springer,
  Berlin–Heidelberg–New York, 1960), \emph{Handbuch der Physik, Fl\"ugge,
  S. (ed)}, vol. 2/3/1, pp. 226--793

\bibitem{Truesdell:1965}
C.~Truesdell, W.~Noll, \emph{The Non-Linear Field Theories of Mechanics}
  (Springer--Verlag, Berlin, Heidelberg, 1965)

\bibitem{Meyers:1994}
M.A. Meyers, \emph{Dynamic behavior of materials} (John Wiley \& Sons, New
  York, 1994)

\bibitem{Bower:2010}
A.F. Bower, \emph{Applied mechanics of solids} (CRC Press, Boca Raton, Fla.,
  2010)

\bibitem{Dashti:2017}
M.~Dashti, A.M. Stuart, \emph{The Bayesian Approach to Inverse Problems}
  (Springer International Publishing, Cham, 2017), pp. 311--428

\bibitem{Bock:2019}
F.E. Bock, R.C. Aydin, C.J. Cyron, N.~Huber, S.R. Kalidindi, B.~Klusemann,
  Frontiers in Materials \textbf{6}, 110 (2019)

\bibitem{conti2018data}
S.~Conti, S.~M{\"u}ller, M.~Ortiz, Archive for Rational Mechanics and Analysis
  \textbf{229}(1), 79 (2018)

\bibitem{kirchdoerfer2016data}
T.~Kirchdoerfer, M.~Ortiz, Computer Methods in Applied Mechanics and
  Engineering \textbf{304}, 81 (2016)

\bibitem{Conti:2020}
S.~Conti, S.~M\"uller, M.~Ortiz, Archive for Rational Mechanics and Analysis
  \textbf{237}(1), 1 (2020)

\bibitem{Roger:2020}
M.~R\"oger, B.~Schweizer, Calculus of Variations and Partial Differential
  Equations \textbf{59}(4), 119 (2020)

\bibitem{Nguyen:2018}
L.T.K. Nguyen, M.A. Keip, Computers \& Structures \textbf{194}, 97 (2018)

\bibitem{Ayensa:2018}
J.~Ayensa-Jim\'enez, M.H. Doweidar, J.A. Sanz-Herrera, M.~Doblar\'e, Computer
  Methods in Applied Mechanics and Engineering \textbf{328}, 752 (2018)

\bibitem{Leygue:2018}
A.~Leygue, M.~Coret, J.~R\'ethor\'e, L.~Stainier, E.~Verron, Computer Methods
  in Applied Mechanics and Engineering \textbf{331}, 184 (2018)

\bibitem{Kanno:2018b}
Y.~Kanno, Japan Journal of Industrial and Applied Mathematics \textbf{35}(3),
  1085 (2018)

\bibitem{Zhou:2020}
Y.~Zhou, H.~Zhan, W.~Zhang, J.~Zhu, J.~Bai, Q.~Wang, Y.~Gu, Computers \&
  Structures \textbf{239}, 106310 (2020)

\bibitem{Gebhardt:2020a}
C.G. Gebhardt, M.C. Steinbach, D.~Schillinger, R.~Rolfes, International Journal
  for Numerical Methods in Engineering \textbf{121}(24), 5447 (2020)

\bibitem{Gebhardt:2020b}
C.G. Gebhardt, D.~Schillinger, M.C. Steinbach, R.~Rolfes, Computer Methods in
  Applied Mechanics and Engineering \textbf{365}, 112993 (2020)

\bibitem{kirchdoerfer2017data}
T.~Kirchdoerfer, M.~Ortiz, Computer Methods in Applied Mechanics and
  Engineering \textbf{326}, 622 (2017)

\bibitem{Federer:1969}
H.~Federer, \emph{Geometric measure theory}, \emph{Die Grundlehren der
  mathematischen Wissenschaften}, vol. 153 (Springer-Verlag, New York, 1969)

\bibitem{mattila1984hausdorff}
P.~Mattila, Acta Mathematica \textbf{152}(1), 77 (1984)

\bibitem{Donsker:1975a}
M.D. Donsker, S.R.S. Varadhan, Communications on Pure and Applied Mathematics
  \textbf{28}, 1 (1975)

\bibitem{leonard2014some}
C.~L{\'e}onard, in \emph{S{\'e}minaire de Probabilit{\'e}s XLVI} (Springer,
  2014), pp. 207--230

\end{thebibliography}
\bibliographystyle{unsrt}

\end{document}